\documentclass[12pt]{amsart}
\usepackage{}

\usepackage{amsmath}
\usepackage{amsfonts}
\usepackage{amssymb}
\usepackage[all]{xy}           

\usepackage{bbding}
\usepackage{txfonts}
\usepackage{amscd}

\usepackage[shortlabels]{enumitem}
\usepackage{ifpdf}
\ifpdf
  \usepackage[colorlinks,final,backref=page,hyperindex]{hyperref}
\else
  \usepackage[colorlinks,final,backref=page,hyperindex,hypertex]{hyperref}
\fi
\usepackage{tikz}
\usepackage[active]{srcltx}

\makeatletter

\newtheorem{thm}{Theorem}[section]

\newtheorem{cor}[thm]{Corollary}
\newtheorem{pro}[thm]{Proposition}

\newtheorem{rmk}[thm]{Remark}
\newtheorem{defi}[thm]{Definition}

\newcommand {\emptycomment}[1]{}

\newcommand{\lon }{\,\rightarrow\,}
\newcommand{\be }{\begin{equation}}
\newcommand{\ee }{\end{equation}}

\newcommand{\g}{\mathfrak g}

\newcommand{\huaB}{\mathcal{B}}

\newcommand{\huaA}{\mathcal{A}}
\newcommand{\huaL}{\mathcal{L}}
\newcommand{\huaR}{\mathcal{R}}

\newcommand{\huaO}{{\mathcal{O}}}

\newcommand{\frks}{\mathfrak s}

\newcommand{\Id}{\rm{Id}}

\newcommand{\br}[1]{   [ \cdot,    \cdot  ]   }

\newcommand{\gl}{\mathfrak {gl}}

\newcommand{\ad}{\mathrm{ad}}

\begin{document}

\title[On Hom-pre-Lie bialgebras]{On Hom-pre-Lie bialgebras
}

\author{Shanshan Liu}
\address{Department of Mathematics, Jilin University, Changchun 130012, Jilin, China}
\email{shanshan18@mails.jlu.edu.cn}

\author{Abdenacer Makhlouf}
\address{University of Haute Alsace, IRIMAS- D\'epartement  de Math\'ematiques, Universit\'e de Haute Alsace, France}
\email{abdenacer.makhlouf@uha.fr}

\author{Lina Song}
\address{Department of Mathematics, Jilin University, Changchun 130012, Jilin, China}
\email{songln@jlu.edu.cn}


\begin{abstract}

In this paper we introduce the notion of  Hom-pre-Lie bialgebra in the general framework of the cohomology theory for Hom-Lie algebras. We show that  Hom-pre-Lie bialgebras, standard Manin triples for Hom-pre-Lie algebras and certain matched pairs of Hom-pre-Lie algebras are equivalent. Due to the usage of the cohomology theory, it makes us successfully study the coboundary Hom-pre-Lie bialgebras. The notion of  Hom-$\frks$-matrix is introduced, by which we can construct Hom-pre-Lie bialgebras naturally. Finally we introduce the notions of Hom-$\huaO$-operators on Hom-pre-Lie algebras and Hom-L-dendriform algebras, by which we construct Hom-$\frks$-matrices.
\end{abstract}


\maketitle
\vspace{-5mm}
\tableofcontents

\allowdisplaybreaks


\section{Introduction}

For a given algebraic structure determined by a set of multiplications of various arities and a set of relations among the operations, a bialgebra structure on this algebra is obtained
by a corresponding set of comultiplications together with a set of compatibility conditions between the multiplications and comultiplications. A good compatibility condition is prescribed   by a rich structure theory and effective constructions.
The most famous examples of bialgebras are  associative bialgebras and Lie bialgebras, which have important applications in both mathematics and mathematical physics.

\vspace{2mm}

The notion of a Hom-Lie algebra was introduced by Hartwig, Larsson
and Silvestrov in \cite{HLS} as part of a study of deformations of
the Witt and the Virasoro algebras. In a Hom-Lie algebra, the Jacobi
identity is twisted by a linear map (homomorphism), and it is called  Hom-Jacobi identity. Different types of Hom-algebras were introduced and widely studied. Recently, in \cite{ELMS}, Elchinger,  Lundengard,  Makhlouf and  Silvestrov extend the result in \cite{HLS} to the case of $(\sigma,\tau)$-derivations. Due to the importance of the aforementioned bialgebra theory, the bialgebra theory for Hom-algebras was deeply studied in \cite{Cai-Sheng,MS3,sheng1,TY,Yao1,Yao3}

\emptycomment{
In particular, Hom-structures on semisimple Lie algebras were studied in \cite{Jin1,Jin2}; geometrization of Hom-Lie algebras were studied in \cite{LGT,CLS};
In \cite{sheng1}, the author gave the definition of an admissible Hom-Lie algebra and introduced a new definition of a Hom-Lie bialgebra. In \cite{Cai-Sheng}, the author constructed the dual representation of a representation of a Hom-Lie algebra and introduced the notation of a purely Hom-Lie bialgebra. In \cite{TY}, the author studied Hom-Lie bialgebras by a new notion of the dual representation of a representation of a Hom-Lie algebra, motivated by the essential connection between Hom-Lie bialgebras and Manin triples. In \cite{Yao1}, the author studied a twisted version of the Yang-Baxter Equation, called the Hom-Yang-Baxter Equation (HYBE), and introduced quasi-triangular bialgebras. In \cite{Yao3}, the author introduced the classical Hom-Yang-Baxter equation (CHYBE) and the closely related structure of Hom-Lie bialgebras, which generalize Drinfel¡¯d¡¯s Lie bialgebras.
}

Pre-Lie algebras (also called  left-symmetric algebras,
quasi-associative algebras, Vinberg algebras and so on) are a
class of nonassociative algebras that  appeared in
many fields in mathematics and mathematical physics.
See  the survey \cite{Pre-lie algebra in geometry} and the references therein for more details. The notion of  left-symmetric bialgebra was introduced in \cite{Bai}. The author also introduced the notion of   $\frks$-matrices   in \cite{Bai} to produce left-symmetric bialgebras.
The notion of a Hom-pre-Lie algebra was introduced in \cite{MS2}  and play important roles in the study of Hom-Lie bialgebras and Hom-Lie 2-algebras \cite{sheng1,SC}. Recently, Hom-pre-Lie algebras were studied from several aspects. The geometrization of Hom-pre-Lie algebras was studied in \cite{Qing}; universal $\alpha$-central extensions of Hom-pre-Lie algebras were studied in \cite{sunbing}.

The purpose of this paper is to give a systematic study of the bialgebra theory for Hom-pre-Lie algebras. Note that the notion of a Hom-pre-Lie bialgebra was already introduced in \cite{QH} under the terminology of Hom-left-symmetric algebra. However the bialgebra structure given in \cite{QH} does not enjoy a coboundary theory. This is also one of our motivation to study Hom-pre-Lie bialgebras that enjoy a rich structure theory. Our Hom-pre-Lie bialgebra structure enjoy the following properties:
\begin{itemize}
  \item  equivalent to a Manin triple for Hom-pre-Lie algebras as well as certain matched pair of Hom-pre-Lie algebras;
  \item Hom-$\frks$-matrices can be defined to produce Hom-pre-Lie bialgebras;
  \item $\huaO$-operators on Hom-pre-Lie algebras can be defined to give Hom-$\frks$-matrices in the semidirect product Hom-pre-Lie algebras.
\end{itemize}

The paper is organized as follows. In Section 2, we recall relevant definitions and results about matched pairs of Hom-Lie algebras and matched pairs of Hom-pre-Lie algebras. In Section 3, we introduce the notion of  Hom-pre-Lie bialgebra and show that it is equivalent to Manin triples as well as matched pairs. In Section 4, we study coboundary Hom-pre-Lie bialgebras and introduce the notion of a Hom-$\frks$-matrix, by which we can construct a Hom-pre-Lie bialgebra naturally. In Section 5, we introduce the notion of an $\huaO$-operator and the notion of a Hom-L-dendriform algebra, by which we can construct Hom-$\frks$-matrices.

\vspace{2mm}
\noindent
{\bf Acknowledgements. }  We give warmest thanks to Yunhe Sheng for helpful comments.

\section{Preliminaries}

In this section, we briefly recall matched pairs of Hom-Lie algebras and matched pairs of Hom-pre-Lie algebras, as preparation for our later study of Hom-pre-Lie bialgebras.
\begin{defi}{\rm(\cite{HLS})}
A {\bf Hom-Lie algebra} is a triple $(\g,[\cdot,\cdot]_\g,\phi_\g)$ consisting of a linear space $\g$, a skew-symmetric bilinear map $[\cdot,\cdot]_\g:\wedge^2\g\longrightarrow \g$ and an algebra morphism $\phi_\g:\g\longrightarrow \g $, satisfying:
  \begin{equation}
[\phi_\g(x),[y,z]_\g]_\g+[\phi_\g(y),[z,x]_\g]_\g+[\phi_\g(z),[x,y]_\g]_\g=0,\quad
\forall~x,y,z\in \g.
\end{equation}
A Hom-Lie algebra $(\g,[\cdot,\cdot]_\g,\phi_\g)$ is said to be regular if $\phi_\g$ is invertible.
  \end{defi}

\begin{defi}{\rm(\cite{AEM,sheng3})}\label{defi:hom-lie representation}
 A {\bf representation} of a Hom-Lie algebra $(\g,[\cdot,\cdot]_\g,\phi_\g)$ on
 a vector space $V$ with respect to $\beta\in\gl(V)$ is a linear map
  $\rho:\g\longrightarrow \gl(V)$, such that for all
  $x,y\in \g$, the following equalities are satisfied:
\begin{eqnarray}
\label{hom-lie-rep-1}\rho(\phi_\g(x))\circ \beta&=&\beta\circ \rho(x),\\
\label{hom-lie-rep-2}\rho([x,y]_\g)\circ \beta&=&\rho(\phi_\g(x))\circ\rho(y)-\rho(\phi_\g(y))\circ\rho(x).
\end{eqnarray}
  \end{defi}
We denote a representation by $(V,\beta,\rho)$. For all $x\in\mathfrak{g}$, we define $\ad_{x}:\mathfrak{g}\lon \mathfrak{g}$ by
\begin{eqnarray}
\ad_{x}(y)=[x,y]_\g,\quad\forall y \in \mathfrak{g}.
\end{eqnarray}
Then $\ad:\g\longrightarrow\gl(\frak g)$ is a representation of the Hom-Lie algebra $(\mathfrak{g},[\cdot,\cdot]_\g,\phi_\g)$ on $\g$ with respect to $\phi_\g$, which is called the adjoint representation.
\begin{pro}
Let $(\g,[\cdot,\cdot]_\g,\phi_\g)$ be a Hom-Lie algebra, $(V,\beta_V,\rho_V)$ and $(W,\beta_W,\rho_W,)$ its representations. Then $(V\otimes W,\beta_V\otimes \beta_W,\rho_V\otimes \beta_W+\beta_V\otimes \rho_W)$ is a representation of $(\g,[\cdot,\cdot]_\g,\phi_\g)$.
\end{pro}
\begin{defi}{\rm(\cite{sheng1})}
A {\bf matched pair of Hom-Lie algebras}, which is denoted by $(\g,\g',\rho,\rho')$, consists of two Hom-Lie algebras $(\g,[\cdot,\cdot]_\g,\phi_\g)$ and $(\g',[\cdot,\cdot]_{\g'},\phi_{\g'})$, together with representations $\rho:\g\longrightarrow \gl(g')$ and $\rho':\g'\longrightarrow \gl(g)$ with respect to $\phi_{\g'}$ and $\phi_\g$ respectively, such that for all $ x,y \in \g, x',y'\in \g'$ the following conditions are satisfied:
\begin{eqnarray}
  \label{matche-pair-1}\rho'(\phi_\g'(x'))[x,y]_\g&=&[\rho'(x')(x),\phi_\g(y)]_\g+[\phi_\g(x),\rho'(x')(y)]_\g\\
  \nonumber &&+\rho'(\rho(y)(x'))(\phi_\g(x))-\rho'(\rho(x)(x'))(\phi_\g(y)),\\
  \label{matche-pair-2}\rho(\phi_\g(x))[x',y']_{\g'}&=&[\rho(x)(x'),\phi_{\g'}(y')]_{\g'}+[\phi_{\g'}(x'),\rho(x)(y')]_{\g'}\\
  \nonumber &&+\rho(\rho'(y')(x))(\phi_{\g'}(x'))-\rho(\rho'(x')(x))(\phi_{\g'}(y')).
\end{eqnarray}
\end{defi}
We define $\phi_d:\g\oplus \g'\longrightarrow \g\oplus \g'$ by
\begin{equation}
\phi_d(x,x')=(\phi_\g(x),\phi_{\g'}(x')),
\end{equation}
and define a skew-symmetric bilinear map $[\cdot,\cdot]_d:\wedge^2(\g\oplus \g')\longrightarrow\g\oplus \g'$ by
\begin{equation}
[(x,x'),(y,y')]_d=([x,y]_\g+\rho'(x')(y)-\rho'(y')(x),[x',y']_{\g'}+\rho(x)(y')-\rho(y)(x')).
\end{equation}
\begin{thm}{\rm(\cite{sheng1})}
With the above notations, $(\g\oplus \g',[\cdot,\cdot]_d,\phi_d)$ is a Hom-Lie algebra if and only if $(\g,\g',\rho,\rho')$ is a matched pair of Hom-Lie algebras.
\end{thm}

\begin{defi}{\rm(\cite{MS2})}
A {\bf Hom-pre-Lie algebra} $(A,\cdot,\alpha)$ is a vector space $A$ equipped with a bilinear product $\cdot:A\otimes A\longrightarrow A$, and $\alpha\in \gl(A)$, such that for all $x,y,z\in A$, $\alpha(x \cdot y)=\alpha(x)\cdot \alpha(y)$ and the following equality is satisfied:
\begin{eqnarray}
(x\cdot y)\cdot \alpha(z)-\alpha(x)\cdot (y\cdot z)=(y\cdot x)\cdot \alpha(z)-\alpha(y)\cdot (x\cdot z).
\end{eqnarray}
A Hom-pre-Lie algebra $(A,\cdot,\alpha)$ is said to be regular if $\alpha$ is invertible.
\end{defi}
Let $(A,\cdot,\alpha)$ be a Hom-pre-Lie algebra. We always assume that it is regular, i.e. $\alpha$ is invertible. The commutator $[x,y]=x\cdot y-y\cdot x$ gives a Hom-Lie algebra $(A,[\cdot,\cdot],\alpha)$, which is denoted by $A^C$ and called the sub-adjacent Hom-Lie algebra of $(A,\cdot,\alpha)$.
\begin{defi}{\rm(\cite{LSS})}
 A {\bf morphism} from a Hom-pre-Lie algebra $(A,\cdot,\alpha)$ to a Hom-pre-Lie algebra $(A',\cdot',\alpha')$ is a linear map $f:A\longrightarrow A'$ such that for all
  $x,y\in A$, the following equalities are satisfied:
\begin{eqnarray}
\label{homo-1}f(x\cdot y)&=&f(x)\cdot' f(y),\hspace{3mm}\forall x,y\in A,\\
\label{homo-2}f\circ \alpha&=&\alpha'\circ f.
\end{eqnarray}
\end{defi}

\begin{defi}{\rm(\cite{QH})}\label{defi:hom-pre representation}
 A {\bf representation} of a Hom-pre-Lie algebra $(A,\cdot,\alpha)$ on a vector space $V$ with respect to $\beta\in\gl(V)$ consists of a pair $(\rho,\mu)$, where $\rho:A\longrightarrow \gl(V)$ is a representation of the sub-adjacent Hom-Lie algebra $A^C$ on $V$ with respect to $\beta\in\gl(V)$, and $\mu:A\longrightarrow \gl(V)$ is a linear map, for all $x,y\in A$, satisfying:
\begin{eqnarray}
\label{rep-1}\beta\circ \mu(x)&=&\mu(\alpha(x))\circ \beta,\\
\label{rep-2}\mu(\alpha(y))\circ\mu(x)-\mu(x\cdot y)\circ \beta&=&\mu(\alpha(y))\circ\rho(x)-\rho(\alpha(x))\circ\mu(y).
\end{eqnarray}
\end{defi}
We denote a representation of a Hom-pre-Lie algebra $(A,\cdot,\alpha)$ by $(V,\beta,\rho,\mu)$. Furthermore, Let $L,R:A\longrightarrow \gl(A)$ be linear maps, where $L_xy=x\cdot y,\quad R_xy=y\cdot x$. Then $(A,\alpha,L,R)$ is also a representation, which we call the regular representation.
\begin{pro}
Let $(A,\cdot,\alpha)$ be a Hom-pre-Lie algebra. For any integer $s$, define $L^s,R^s:A\longrightarrow \gl(A)$ by
$$L^s_xy=\alpha^s(x)\cdot y,\quad R^s_xy=y\cdot \alpha^s(x), \quad \forall x,y \in A.$$
Then $(A,\alpha,L^s,R^s)$ is a representation of a Hom-Lie algebra $(A,\cdot,\alpha)$.
\end{pro}
\begin{proof}
For all $x,y,z\in A$, we have
\begin{eqnarray}
L^s_{\alpha(x)}\alpha(y)=\alpha^{s+1}(x)\cdot \alpha(y)=\alpha(\alpha^s(x)\cdot y)=\alpha(L^s_xy),
\end{eqnarray}
which implies that $L^s_{\alpha(x)}\circ \alpha=\alpha\circ L^s_x$. Similarly, we have $R^s_{\alpha(x)}\circ \alpha=\alpha\circ R^s_x$.

By the definition of a Hom-pre-Lie algebra, we have
\begin{eqnarray}
\nonumber L^s_{\alpha(x)}L^s_y(z)-L^s_{\alpha(y)}L^s_x(z)
\nonumber&=&\alpha^{s+1}(x)\cdot(\alpha^s(y)\cdot z)-\alpha^{s+1}(y)\cdot(\alpha^s(x)\cdot z)\\
\nonumber&=&(\alpha^s(x)\cdot \alpha^s(y))\cdot\alpha(z)-(\alpha^s(y)\cdot \alpha^s(x))\cdot\alpha(z)\\
\nonumber &=&\alpha^s([x,y])\cdot\alpha(z)\\
\nonumber &=&L^s_{[x,y]}\alpha(z).
\end{eqnarray}
Similarly, we have
\begin{equation}
R^s_{\alpha(y)}\circ R^s_x-R^s_{x\cdot y}\circ \alpha=R^s_{\alpha(y)}\circ L^s_x-L^s_{\alpha(x)}\circ R^s_y.
\end{equation}
This finishes the proof.
\end{proof}

The notion of matched pair of Hom-pre-Lie algebras was given in \cite{QH}.

\begin{defi}{\rm(\cite{QH})}
A {\bf matched pair of Hom-pre-Lie algebras} $(A,B,l_A,r_A,l_B,r_B)$ consists of two Hom-pre-Lie algebras $(A,\cdot,\alpha_A)$ and $(B,\circ,\alpha_B)$, together with linear maps $l_A,r_A:A\longrightarrow \gl(B)$ and $l_B,r_B:B\longrightarrow \gl(A)$ such that $(B,\alpha_B,l_A,r_A)$ and $(A,\alpha_A,l_B,r_B)$ are representations and for all $x,y\in A, a,b\in B$, satisfying the following conditions:
\begin{eqnarray}
  \label{pre-matche-pair-1}r_A(\alpha_A(x))\{a,b\}&=&r_A(l_B(b)x)\alpha_B(a)-r_A(l_B(a)x)\alpha_B(b)+\alpha_B(a)\circ(r_A(x)b)\\
  \nonumber &&-\alpha_B(b)\circ(r_A(x)a),\\
  \label{pre-matche-pair-2}l_A(\alpha_A(x))(a\circ b)&=&-l_A(l_B(a)x-r_B(a)x)\alpha_B(b)+(l_A(x)a-r_A(x)a)\circ \alpha_B(b)\\
  \nonumber &&+r_A(r_B(b)x)\alpha_B(a)+\alpha_B(a)\circ(l_A(x)b),\\
  \label{pre-matche-pair-3}r_B(\alpha_B(a))[x,y]&=&r_B(l_A(y)a)\alpha_A(x)-r_B(l_A(x)a)\alpha_A(y)+\alpha_A(x)\cdot(r_B(a)y)\\
  \nonumber &&-\alpha_A(y)\cdot(r_B(a)x),\\
  \label{pre-matche-pair-4}l_B(\alpha_B(a))(x\cdot y)&=&-l_B(l_A(x)a-r_A(x)a)\alpha_A(y)+(l_B(a)x-r_B(a)x)\cdot \alpha_A(y)\\
  \nonumber &&+r_B(r_A(y)a)\alpha_A(x)+\alpha_A(x)\cdot(l_B(a)y),
\end{eqnarray}
where $[\cdot,\cdot]$ is the Lie bracket of the sub-adjacent Hom-Lie algebra $A^C$ and $\{\cdot,\cdot\}$ is the Lie bracket of the sub-adjacent Hom-Lie algebra $B^C.$
\end{defi}
We define a bilinear operation $\diamond:\otimes^2(A\oplus B)\lon(A\oplus B)$ by
\begin{equation}\label{Hom-pre-lie matched pair}
(x+a)\diamond(y+b):=x\cdot y+l_B(a)y+r_B(b)x+a\circ b+l_A(x)b+r_A(y)a,
\end{equation}
and a linear map $\alpha_A\oplus \alpha_B:A\oplus B\longrightarrow A\oplus B$ by
\begin{equation}
(\alpha_A\oplus \alpha_B)(x+a):=\alpha_A(x)+\alpha_B(a).
\end{equation}
The following is  proved in \cite{QH}.
\begin{thm}{\rm(\cite{QH})}\label{matched-pair-Hom-pre-Lie algebra}
With the above notations, $(A\oplus B,\diamond,\alpha_A\oplus \alpha_B)$ is a Hom-pre-Lie algebra if and only if $(A,B,l_A,r_A,l_B,r_B)$ is a matched pair of Hom-pre-Lie algebras.
\end{thm}
\section{Hom-pre-Lie bialgebras}

In this section, we introduce the notions of Manin triple for Hom-pre-Lie algebras and Hom-pre-Lie bialgebra. We show that  Hom-pre-Lie bialgebras, standard Manin triples for Hom-pre-Lie algebras and certain matched pairs of Hom-pre-Lie algebras are equivalent.

Let $(V,\beta,\rho,\mu)$ be a representation of a Hom-pre-Lie algebra $(A,\cdot,\alpha)$. In the sequel, we always assume that $\beta$ is invertible. For all $x\in A,u\in V,\xi\in V^*$, define $\rho^*:A\longrightarrow\gl(V^*)$ and $\mu^*:A\longrightarrow\gl(V^*)$ as usual by
$$\langle \rho^*(x)(\xi),u\rangle=-\langle\xi,\rho(x)(u)\rangle,\quad \langle \mu^*(x)(\xi),u\rangle=-\langle\xi,\mu(x)(u)\rangle,\quad\forall x\in A,~\xi\in V^*,~u\in V.$$
Then define $\rho^\star:A\longrightarrow\gl(V^*)$ and $\mu^\star:A\longrightarrow\gl(V^*)$ by
\begin{eqnarray}
  \label{eq:1.3}\rho^\star(x)(\xi):=\rho^*(\alpha(x))\big{(}(\beta^{-2})^*(\xi)\big{)},\\
   \label{eq:1.4}\mu^\star(x)(\xi):=\mu^*(\alpha(x))\big{(}(\beta^{-2})^*(\xi)\big{)}.
\end{eqnarray}
\begin{thm}{\rm(\cite{LSS})}\label{dual-rep}
Let $(V,\beta,\rho,\mu)$ be a representation of a Hom-pre-Lie algebra $(A,\cdot,\alpha)$. Then $(V^*,(\beta^{-1})^*,\rho^\star-\mu^\star,-\mu^\star)$ is a representation of $(A,\cdot,\alpha)$, which is called the dual representation of $(V,\beta,\rho,\mu)$.
\end{thm}

\begin{cor}
 Let  $(A,\cdot,\alpha)$ be a Hom-pre-Lie algebra. Then $(A^*,(\alpha^{-1})^*,\ad^\star,-R^\star)$ is a representation of $(A,\cdot,\alpha)$, where $\ad=L-R$ is the adjoint representation of the sub-adjacent Hom-Lie algebra $A^C.$
\end{cor}

In the sequel, when there is a Hom-pre-Lie algebra structure on $A^*$, we will use $\huaL,~\huaR$ and $\mathfrak{ad}:=\huaL-\huaR$ to denote the corresponding operations.

Before we introduce the notions of Manin triple and Hom-pre-Lie bialgebra, we give an important relation  between matched pairs of Hom-pre-Lie algebras and matched pairs of the associated sub-adjacent Hom-Lie algebras.

\begin{pro}\label{matched-pair-equivalent}
Let $(A,\cdot,\alpha)$ and $(A^\ast,\circ,(\alpha^{-1})^\ast)$ be Hom-pre-Lie algebras. Then $(A^C,(A^\ast)^C,L^\star,\huaL^\star)$ is a matched pair of Hom-Lie algebras if and only if $(A,A^\ast,\ad^\star,-R^\star,\mathfrak{ad}^\star,-\huaR^\star)$ is a matched pair of Hom-pre-Lie algebras.
\end{pro}
\begin{proof}
We denote the bracket of the sub-adjacent Hom-pre-Lie algebra $(A^*)^C$ by $\{\cdot,\cdot\}$. For all $x,y\in A, \xi,\eta\in A^\ast$, we have
\begin{eqnarray}\label{matched-1}
\nonumber &&\langle R^\star_{\alpha(\alpha^{-2}(x))}\{\xi,\eta\}-R^\star_{\mathfrak{ad}^\star_\eta\alpha^{-2}(x)}(\alpha^{-1})^\ast(\xi)+R^\star_{\mathfrak{ad}^\star_\xi\alpha^{-2}(x)}(\alpha^{-1})^\ast(\eta)-(\alpha^{-1})^\ast(\xi)\circ R^\star_{\alpha^{-2}(x)}\eta\\
\nonumber &&+(\alpha^{-1})^\ast(\eta)\circ R^\star_{\alpha^{-2}(x)}\xi,\alpha^2(y)\rangle\\
\nonumber&=&-\langle\{\xi,\eta\},R_{\alpha^{-2}(x)}y\rangle+\langle (\alpha^{-1})^\ast(\xi),R_{\alpha^{-1}(\mathfrak{ad}^\star_\eta\alpha^{-2}(x))}(y)\rangle-\langle (\alpha^{-1})^\ast(\eta),R_{\alpha^{-1}(\mathfrak{ad}^\star_\xi\alpha^{-2}(x))}(y)\rangle\\
\nonumber &&-\langle \huaL_{(\alpha^{-1})^\ast(\xi)}R^\star_{\alpha^{-2}(x)}\eta,\alpha^2(y)\rangle+\langle \huaL_{(\alpha^{-1})^\ast(\eta)}R^\star_{\alpha^{-2}(x)}\xi,\alpha^2(y)\rangle\\
\nonumber &=&\langle-\{\xi,\eta\},L_y\alpha^{-2}(x)\rangle+\langle(\alpha^{-1})^\ast(\xi),L_y\alpha^{-1}(\mathfrak{ad}^\star_\eta\alpha^{-2}(x))\rangle-\langle(\alpha^{-1})^\ast(\eta),L_y\alpha^{-1}(\mathfrak{ad}^\star_\xi\alpha^{-2}(x))\rangle\\
\nonumber &&+\langle R^\star_{\alpha^{-2}(x)}\eta,\huaL^\ast_{(\alpha^{-1})^\ast(\xi)}\alpha^2(y)\rangle-\langle R^\star_{\alpha^{-2}(x)}\xi,\huaL^\ast_{(\alpha^{-1})^\ast(\eta)}\alpha^2(y)\rangle\\
\nonumber  &=&\langle L^\star_{\alpha(y)}\{\xi,\eta\},x\rangle-\langle(\alpha^{-1})^\ast L^\ast_y(\alpha^{-1})^\ast(\xi),\mathfrak{ad}^\star_\eta\alpha^{-2}(x)\rangle+\langle(\alpha^{-1})^\ast L^\ast_y(\alpha^{-1})^\ast(\eta),\mathfrak{ad}^\star_\xi\alpha^{-2}(x)\rangle\\
\nonumber &&-\langle\eta,R_{\alpha^{-3}(x)}\alpha^{-2}(\huaL^\ast_{(\alpha^{-1})^\ast(\xi)}\alpha^2(y))\rangle+\langle\xi,R_{\alpha^{-3}(x)}\alpha^{-2}(\huaL^\ast_{(\alpha^{-1})^\ast(\eta)}\alpha^2(y))\rangle\\
\nonumber&=&\langle L^\star_{\alpha(y)}\{\xi,\eta\},x\rangle+\langle \mathfrak{ad}_{\alpha^\ast(\eta)}\alpha^\ast (L^\ast_y(\alpha^{-1})^\ast(\xi)),\alpha^{-2}(x)\rangle
-\langle \mathfrak{ad}_{\alpha^\ast(\xi)}\alpha^\ast (L^\ast_y(\alpha^{-1})^\ast(\eta)),\alpha^{-2}(x)\rangle\\
\nonumber &&+\langle L^\ast_{\alpha^{-2}(\huaL^\star_\xi y)}\eta,\alpha^{-3}(x)\rangle-\langle L^\ast_{\alpha^{-2}(\huaL^\star_\eta y)}\xi,\alpha^{-3}(x)\rangle\\
\nonumber &=&\langle L^\star_{\alpha(y)}\{\xi,\eta\},x\rangle+\langle \{\eta,L^\star_{\alpha^{-1}(y)}\alpha^\ast(\xi)\},\alpha^{-1}(x)\rangle-\langle \{\xi,L^\star_{\alpha^{-1}(y)}\alpha^\ast(\eta)\},\alpha^{-1}(x)\rangle\\
\nonumber &&+\langle L^\star_{\alpha^{-3}(\huaL^\star_\xi y)}(\alpha^2)^\ast(\eta),\alpha^{-3}(x)\rangle-\langle L^\star_{\alpha^{-3}(\huaL^\star_\eta y)}(\alpha^2)^\ast(\xi),\alpha^{-3}(x)\rangle\\
\nonumber &=&\langle L^\star_{\alpha(y)}\{\xi,\eta\},x\rangle+\langle \{\eta,\alpha^\ast (L^\star_y \xi)\},\alpha^{-1}(x)\rangle-\langle \{\xi,\alpha^\ast (L^\star_y \eta)\},\alpha^{-1}(x)\rangle+\langle L^\star_{\huaL^\star_\xi y}(\alpha^{-1})^\ast(\eta),x\rangle\\
\nonumber &&-\langle L^\star_{\huaL^\star_\eta y}(\alpha^{-1})^\ast(\xi),x\rangle\\
\nonumber&=&\langle L^\star_{\alpha(y)}\{\xi,\eta\}+\{(\alpha^{-1})^\ast(\eta),L^\star_y \xi\}-\{(\alpha^{-1})^\ast(\xi),L^\star_y \eta)\}+L^\star_{\huaL^\star_\xi y}(\alpha^{-1})^\ast(\eta)-L^\star_{\huaL^\star_\eta y}(\alpha^{-1})^\ast(\xi),x\rangle,
\end{eqnarray}
which implies that $\eqref{matche-pair-2}\Longleftrightarrow \eqref{pre-matche-pair-1}$. We also have
\begin{eqnarray} 
\nonumber &&\langle-\mathfrak{ad}^\star_{(\alpha^{-1})^\ast(\xi)}(x\cdot y)-\mathfrak{ad}^\star_{L^\star_x \xi}\alpha(y)+\huaL^\star_{\xi}x \cdot\alpha(y)+\huaR^\star_{R^\star_y\xi}\alpha(x)+\alpha(x)\cdot \mathfrak{ad}^\star_\xi y,(\alpha^{-2})^\ast(\eta)\rangle\\
\nonumber&=&\langle x\cdot y,\mathfrak{ad}_\xi \eta\rangle+\langle \alpha(y),\mathfrak{ad}_{\alpha^\ast(L^\star_x\xi)}\eta\rangle-\langle\alpha(y),L^\ast_{\huaL^\star_\xi x}(\alpha^{-2})^\ast(\eta)\rangle-\langle \alpha(x),\huaR_{\alpha^\ast(R^\star_y\xi)}\eta\rangle\\
\nonumber &&-\langle \mathfrak{ad}^\star_\xi y,L^\ast_{\alpha(x)}(\alpha^{-2})^\ast(\eta)\rangle\\
\nonumber&=&\langle L_x y,\{\xi,\eta\}\rangle+\langle \alpha(y),\{\alpha^\ast(L^\star_x \xi),\eta\}\rangle-\langle \alpha(y),L^\star_{\alpha^{-1}(\huaL^\star_\xi x)}\eta\rangle-\langle\alpha(x),\huaL_\eta\alpha^\ast(R^\star_y \xi) \rangle-\langle \mathfrak{ad}^\star_\xi y,L^\star_x \eta\rangle\\
\nonumber &=&-\langle \alpha^2(y),L^\star _{\alpha(x)}\{\xi,\eta\}\rangle+\langle \alpha^2(y),\{L^\star_x \xi,(\alpha^{-1})^\ast(\eta)\}\rangle-\langle \alpha^2(y),(\alpha^{-1})^\ast( L^\star_{\alpha^{-1}(\huaL^\star_\xi x)}\eta)\rangle\\
\nonumber &&-\langle R_{\alpha^{-1}(y)}\alpha^{-1}(\huaL^\ast_\eta \alpha(x)),\xi\rangle+\langle y,\mathfrak{ad}_{\alpha^\ast(\xi)}(\alpha^2)^\ast(L^\star_x \eta)\rangle\\
\nonumber &=&-\langle \alpha^2(y),L^\star _{\alpha(x)}\{\xi,\eta\}\rangle+\langle \alpha^2(y),\{L^\star_x \xi,(\alpha^{-1})^\ast(\eta)\}\rangle-\langle \alpha^2(y),L^\star_{\huaL^\star_\xi x}(\alpha^{-1})^\ast(\eta)\rangle\\
\nonumber&& -\langle \alpha^2(\huaL^\ast_\eta \alpha(x))\cdot\alpha^2(y),(\alpha^{-3})^\ast(\xi)\rangle+\langle \alpha^2(y),(\alpha^{-2})^\ast \mathfrak{ad}_{\alpha^\ast(\xi)}(\alpha^2)^\ast(L^\star_x \eta)\rangle\\
\nonumber &=&-\langle \alpha^2(y),L^\star_{\alpha(x)}\{\xi,\eta\}\rangle+\langle \alpha^2(y),\{L^\star_x \xi,(\alpha^{-1})^\ast(\eta)\}\rangle-\langle \alpha^2(y),L^\star_{\huaL^\star_\xi x}(\alpha^{-1})^\ast(\eta)\rangle\\
\nonumber &&-\langle L_{\alpha(\huaL^\star_\eta x)}\alpha^2(y),(\alpha^{-3})^\ast(\xi)\rangle+\langle \alpha^2(y),\{(\alpha^{-1})^\ast(\xi),L^\star_x \eta\}\rangle\\
\nonumber&=&\langle \alpha^2(y),-L^\star _{\alpha(x)}\{\xi,\eta\}+\{L^\star_x \xi,(\alpha^{-1})^\ast(\eta)\}-L^\star_{\huaL^\star_\xi x}(\alpha^{-1})^\ast(\eta)+L^\star_{\huaL^\star_\eta x}(\alpha^{-1})^\ast(\xi)+\{(\alpha^{-1})^\ast(\xi),L^\star_x \eta\}\rangle,
\end{eqnarray}
 which implies that $\eqref{matche-pair-2}\Longleftrightarrow \eqref{pre-matche-pair-4}$.

 Thus, we have $\eqref{matche-pair-2}\Longleftrightarrow\eqref{pre-matche-pair-1} \Longleftrightarrow \eqref{pre-matche-pair-4}$.
 Similarly, we have $\eqref{matche-pair-1}\Longleftrightarrow\eqref{pre-matche-pair-2} \Longleftrightarrow \eqref{pre-matche-pair-3}$. This finishes the proof.
\end{proof}

Now we introduce the notion of quadratic Hom-pre-Lie algebra and the notion of  Manin triple.
\begin{defi}\label{invariant}
A {\bf quadratic Hom-pre-Lie algebra} is a Hom-pre-Lie algebra $(A,\cdot,\alpha)$ equipped with a nondegenerate skew-symmetric bilinear form $\omega\in \wedge^2 A^\ast$ such that for all $x,y,z\in A$, the following invariant conditions hold:
\begin{eqnarray}
\label{invariant-1}\omega(\alpha(x),\alpha(y))&=&\omega(x,y),\\
\label{invariant-2}\omega(x\cdot y,\alpha(z))&=&-\omega(\alpha(y),[x,z]).
\end{eqnarray}
\end{defi}
We denote a quadratic Hom-pre-Lie algebra by $(A,\cdot,\alpha,\omega)$.
\begin{defi}
A {\bf Manin triple} for Hom-pre-Lie algebras is a triple $(\huaA,A_1,A_2)$ in which $(\huaA,\cdot,\alpha,\omega)$ is a quadratic Hom-pre-Lie algebra, $(A_1,\cdot_1,\alpha_1)$ and $(A_2,\cdot_2,\alpha_2)$ are isotropic Hom-pre-Lie sub-algebras of $\huaA$ such that
\begin{itemize}
   \item[$\rm(i)$]  $\huaA=A_1\oplus A_2$ as vector spaces,
  \item[$\rm(ii)$] $\alpha=\alpha_1\oplus \alpha_2.$
\end{itemize}
\end{defi}
Two Manin triples $(\huaA,A_1,A_2)$ and $(\huaB,B_1,B_2)$ with the bilinear forms $\omega_1$ and $\omega_2$ respectively are isomorphic if there exists an isomorphism of Hom-pre-Lie algebras $f:\huaA\longrightarrow \huaB$ such that
\begin{equation}
f(A_1)=B_1, \quad f(A_2)=B_2,\quad \omega_1(x,y)=\omega_2(f(x),f(y)),\quad\forall x,y \in \huaA.
\end{equation}
Let $(A,\cdot,\alpha)$ and $(A^\ast,\circ,(\alpha^{-1})^\ast)$ be two Hom-pre-Lie algebras, for all $x,y\in A, \xi,\eta\in A^*,$ define a bilinear operator $\diamond:\otimes^2(A\oplus A^*)\lon(A\oplus A^*)$ by
\begin{equation}\label{Hom-pre-lie standard-matched pair}
(x+\xi)\diamond(y+\eta):=x\cdot y+\mathfrak{ad}^\star(\xi)y-\huaR^\star(\eta)x+\xi\circ \eta+\ad^\star(x)\eta-R^\star(y)\xi.
\end{equation}
 If $(A\oplus A^*,\diamond,\alpha\oplus (\alpha^{-1})^*)$ is a Hom-pre-Lie algebra, such that $(A,\cdot,\alpha)$ and $(A^\ast,\circ,(\alpha^{-1})^\ast)$ are Hom-pre-Lie subalgebras, by computation, the natural nondegenerate skew-symmetric bilinear form $\bar{\omega}$ on $A\oplus A^\ast$ given by
\begin{equation}\label{standard manin triple}
\bar{\omega}(x+\xi,y+\eta)=\langle \xi,y \rangle-\langle \eta,x \rangle,
\end{equation}
is invariant. Consequently,  $(A\oplus A^\ast,A,A^\ast)$ is a Manin triple, which is called the {\bf standard Manin triple}.
\begin{pro}\label{Manin triple isomorphic}
Every Manin triple is isomorphic to the standard Manin triple.
\end{pro}

\begin{proof}
Let $(\huaA,A_1,A_2)$ be a Manin triple with a nondegenerate skew-symmetric invariant bilinear form $\omega$. For all $x\in A_1, u\in A_2$, define a linear map $f:\huaA\longrightarrow A_1\oplus A_1^*$ by $f(x,u)=(x,\omega(u,\cdot))$. Since $\omega$ is nondegenerate, $f$ is an isomorphism between vector spaces. Thus, $f$ induces a Manin triples structure on $(A_1\oplus A_1^*,A_1,A_1^*)$.

First for all $x,y\in A, u,v\in A_2$, we have
\begin{eqnarray}
\nonumber \bar{\omega}( f(x+u),f(y+v))
=\bar{\omega}(x+\omega(u,\cdot),y+\omega(v,\cdot))=
 \omega(u,y)-\omega(v,x)=\omega(x+u,y+v),
\end{eqnarray}
which implies that the induced bilinear form on $A_1\oplus A_1^* $ is exactly $\bar{\omega}$ given by \eqref{standard manin triple}.

Then we assume the induced Hom-pre-Lie algebra structure on  $A_1\oplus A_1^\ast$ is given by $(A_1\oplus A_1^*,\cdot',\alpha_{A_1}\oplus \alpha_{A_1^*})$. By \eqref{invariant-1}, we obtain $\alpha_{A_1^\ast}=(\alpha_{A_1}^{-1})^\ast$. For all $x,y\in A, \xi,\eta\in A^\ast$, we have
\begin{eqnarray}\label{equivalent-3-4}
\nonumber  \bar{\omega}(x \cdot' \xi,\alpha(y))&=&-\bar{\omega}((\alpha^{-1})^\ast(\xi),[x,y])\\
\nonumber &=&-\langle (\alpha^{-2})^\ast(\xi),L_\alpha(x) \alpha(y)-R_\alpha(x) \alpha(y)\rangle\\
\nonumber &=&\langle L^\ast_{\alpha(x)} (\alpha^{-2})^\ast(\xi)-R^\ast_{\alpha(x)} (\alpha^{-2})^\ast(\xi),\alpha(y)\rangle\\
\nonumber &=& \langle L^\star_x \xi-R^\star_x \xi,\alpha(y)\rangle\\
&=&\bar{\omega}(\ad^\star_x \xi,\alpha(y)),
\end{eqnarray}
and
\begin{eqnarray}\label{equivalent-3-5}
\nonumber  \bar{\omega}(x\cdot'\xi,(\alpha^{-1})^\ast(\eta))&=&\bar{\omega}((\alpha^{-1})^\ast(\xi),[\eta,x])\\
\nonumber &=&-\bar{\omega}(\eta\circ \xi,\alpha(x))\\
\nonumber &=&-\langle \huaR_{(\alpha^{-1})^\ast(\xi)} (\alpha^{-1})^\ast(\eta),\alpha^2(x)\rangle\\
\nonumber &=&\langle (\alpha^{-1})^\ast(\eta),\huaR^\ast_{(\alpha^{-1})^\ast(\xi)} \alpha^2(x)\rangle\\
\nonumber &=& \langle (\alpha^{-1})^\ast(\eta),\huaR^\star_\xi x\rangle\\
&=&-\bar{\omega}(\huaR^\star_\xi x,(\alpha^{-1})^\ast(\eta)).
\end{eqnarray}
Thus, we have $x\cdot'\xi=\ad^\star_x \xi-\huaR^\star_\xi x$. Similarly, we have $\xi\cdot' x=\mathfrak{ad}^\star_\xi x-R^\star_x \xi$. Thus, we  deduce that $(x+\xi)\cdot'(y+\eta)=(x+\xi)\diamond(y+\eta)$, which implies that  $(A_1\oplus A_1^\ast,A_1,A_1^\ast)$ is the standard Manin triple.
\end{proof}
For a Hom-Lie algebra $(\g,[\cdot,\cdot]_\g,\phi_\g)$ and a representation $(V,\beta,\rho)$, recall that a $1$-cocycle $\delta$ associated to $(V,\beta,\rho)$ is a linear map form $\g$
to $V$ satisfying:
\begin{equation}
\delta([x,y]_\g)=\rho(\phi_\g(x))\delta(y)-\rho(\phi_\g(y))\delta(x).
\end{equation}
\begin{defi}
A pair of Hom-pre-Lie algebras $(A,\cdot,\alpha)$ and $(A^\ast,\circ,(\alpha^{-1})^\ast)$   is called a {\bf Hom-pre-Lie bialgebra} if the following conditions hold:
\begin{itemize}
   \item[$\rm(i)$]  $\varphi^\ast$ is a $1$-cocycle of the sub-adjacent Hom-Lie algebra  $A^C$ associated to the representation $(A\otimes A,L^{-2}\otimes\alpha+\alpha \otimes \ad^{-2})$, where $\varphi^\ast:A\longrightarrow A\otimes A$ is the dual of $\circ:A^\ast\otimes A^\ast\longrightarrow A^\ast$, i.e. $\langle \varphi^\ast(x),\xi\otimes\eta\rangle=\langle x,\xi\circ \eta\rangle$.
  \item[$\rm(ii)$] $\psi^\ast$ is a $1$-cocycle of the sub-adjacent Hom-Lie algebra $(A^\ast)^C$ associated to the representation $(A^*\otimes A^*,\huaL^{-2}\otimes(\alpha^{-1})^\ast+(\alpha^{-1})^\ast\otimes \mathfrak{ad}^{-2})$, where $\psi^\ast:A^\ast\longrightarrow A^\ast\otimes A^\ast$ is the dual of $\cdot:A\otimes A\longrightarrow A^\ast$, i.e. $\langle \psi^\ast(\xi),x\otimes y\rangle=\langle \xi,x\cdot y\rangle$.
\end{itemize}
We denote a Hom-pre-Lie bialgebra by $(A,A^\ast,\varphi^*,\psi^*)$ or simply $(A,A^\ast)$.
\end{defi}

Now we are ready to give the main result of this section.
\begin{thm}\label{equivalent}
Let $(A,\cdot,\alpha)$ and $(A^\ast,\circ,(\alpha^{-1})^\ast)$ be two Hom-pre-Lie algebras. Then the following conditions are equivalent:
\begin{itemize}
\item [$\rm(i)$]  $(A,A^\ast)$ is a Hom-pre-Lie bialgebra,
\item[$\rm(ii)$] $(A,A^\ast,\ad^\star,-R^\star,\mathfrak{ad}^\star,-\huaR^\star)$ is a matched pair of Hom-pre-Lie algebras,
\item[$\rm(iii)$] $(A\oplus A^\ast,A,A^\ast)$ is the standard Manin triple for Hom-pre-Lie algebras.
\end{itemize}
\end{thm}
\begin{proof}
First, we prove that $\rm(i)$ is equivalent to $\rm(ii)$. We have
\begin{eqnarray}\label{cocycle-1}
\nonumber && \langle-\ad^\star_{\alpha(x)}(\xi\circ\eta)-\ad^\star_{\huaL^\star_\xi x}(\alpha^{-1})^\ast(\eta)+L^\star_x\xi\circ (\alpha^{-1})^\ast(\eta)+R^\star(\huaR^\star_\eta x)(\alpha^{-1})^\ast(\xi)\\
\nonumber&&+(\alpha^{-1})^\ast(\xi)\circ\ad^\star_x \eta,\alpha^2(y)\rangle\\
\nonumber &=&\langle[x,y],\xi\circ\eta \rangle+\langle \ad_{\alpha^{-1}(\huaL^\star_\xi x)}y,(\alpha^{-1})^\ast(\eta)\rangle+\langle (\alpha^2)^*(L^\star_x \xi)\circ \alpha^*(\eta),y\rangle\\
\nonumber&&-\langle R_{\alpha^{-1}(\huaR^\star_\eta x)}y,(\alpha^{-1})^\ast(\xi)\rangle+\langle \alpha^\ast(\xi)\circ (\alpha^2)^*(\ad^\star_x \eta),y\rangle\\
\nonumber &=&\langle\varphi^\ast[x,y],\xi\otimes \eta\rangle-\langle\ad_y \alpha^{-1}(\huaL^\star_\xi x),(\alpha^{-1})^\ast(\eta)\rangle+\langle L^\star_{\alpha^{-2}(x)}(\alpha^2)^*(\xi)\circ \alpha^*(\eta),y  \rangle\\
\nonumber&&-\langle L_y\alpha^{-1}(\huaR^\star_\eta x),(\alpha^{-1})^\ast(\xi)\rangle+\langle \alpha^\ast(\xi)\circ \ad^\star_{\alpha^{-2}(x)}(\alpha^2)^*(\eta),y \rangle\\
\nonumber &=&\langle \varphi^\ast[x,y],\xi\otimes \eta\rangle+\langle\alpha^{-1}(\huaL^\star_\xi x),\ad^*_y (\alpha^{-1})^\ast(\eta)\rangle+\langle L^*_{\alpha^{-1}(x)}\xi\circ \alpha^*(\eta),y  \rangle\\
\nonumber&&+\langle \alpha^{-1}(\huaR^\star_\eta x),L^*_y(\alpha^{-1})^\ast(\xi) \rangle +\langle \alpha^\ast(\xi)\circ \ad^*_{\alpha^{-1}(x)}\eta,y  \rangle\\
\nonumber &=&\langle \varphi^\ast[x,y],\xi\otimes \eta \rangle+\langle \huaL^\star_\xi x,(\alpha^{-1})^\ast\ad^\star_{\alpha^{-1}(y)}\alpha^\ast(\eta) \rangle+\langle (L^{-2})^*_{\alpha(x)}\xi\circ \alpha^*(\eta),y\rangle\\
\nonumber&&+\langle \huaR^\star_\eta x,(\alpha^{-1})^\ast L^\star_{\alpha^{-1}(y)}\alpha^\ast(\xi)\rangle+\langle \alpha^\ast(\xi)\circ (\ad^{-2})^*_{\alpha(x)}\eta,y \rangle\\
\nonumber &=&\langle \varphi^\ast[x,y],\xi\otimes \eta \rangle+\langle  \huaL^\star_\xi x,\ad^\star_y \eta \rangle-\langle(L_{\alpha(x)}^{-2}\otimes \alpha)\varphi^*(y),\xi\otimes \eta  \rangle+\langle\huaR^\star_\eta x, L^\star_y \xi \rangle\\
\nonumber&&-\langle (\alpha\otimes\ad_{\alpha(x)}^{-2})\varphi^*(y),\xi\otimes \eta\rangle\\
\nonumber &=&\langle \varphi^\ast[x,y],\xi\otimes \eta \rangle-\langle  x,\alpha^*(\xi)\circ (\alpha^2)^*(\ad^\star_y \eta) \rangle-\langle(L_{\alpha(x)}^{-2}\otimes \alpha)\varphi^*(y),\xi\otimes \eta  \rangle\\
\nonumber&&-\langle x, (\alpha^2)^*(L^\star_y \xi)\circ \alpha^*(\eta)\rangle-\langle (\alpha\otimes\ad_{\alpha(x)}^{-2})\varphi^*(y),\xi\otimes \eta\rangle\\
\nonumber &=&\langle \varphi^\ast[x,y],\xi\otimes \eta \rangle-\langle  x,\alpha^*(\xi)\circ \ad^\star_{\alpha^{-2}(y)}(\alpha^2)^*(\eta) \rangle-\langle(L_{\alpha(x)}^{-2}\otimes \alpha)\varphi^*(y),\xi\otimes \eta  \rangle\\
\nonumber&&-\langle x, L^\star_{\alpha^{-2}(y)}(\alpha^2)^*(\xi)\circ \alpha^*(\eta)\rangle-\langle (\alpha\otimes\ad_{\alpha(x)}^{-2})\varphi^*(y),\xi\otimes \eta\rangle\\
\nonumber &=&\langle \varphi^\ast[x,y],\xi\otimes \eta \rangle-\langle  x,\alpha^*(\xi)\circ \ad^*_{\alpha^{-1}(y)}\eta \rangle-\langle(L_{\alpha(x)}^{-2}\otimes \alpha)\varphi^*(y),\xi\otimes \eta  \rangle\\
\nonumber&&-\langle x, L^*_{\alpha^{-1}(y)}\xi\circ \alpha^*(\eta)\rangle-\langle (\alpha\otimes\ad_{\alpha(x)}^{-2})\varphi^*(y),\xi\otimes \eta\rangle\\
\nonumber &=&\langle \varphi^\ast[x,y],\xi\otimes \eta \rangle-\langle  x,\alpha^*(\xi)\circ (\ad^{-2})_{\alpha(y)}^*\eta \rangle-\langle(L_{\alpha(x)}^{-2}\otimes \alpha)\varphi^*(y),\xi\otimes \eta  \rangle\\
\nonumber&&-\langle x, (L^{-2})_{\alpha(y)}^*\xi\circ \alpha^*(\eta)\rangle-\langle (\alpha\otimes\ad_{\alpha(x)}^{-2})\varphi^*(y),\xi\otimes \eta\rangle\\
\nonumber &=&\langle \varphi^\ast[x,y],\xi\otimes \eta \rangle+\langle (\alpha\otimes \ad_{\alpha(y)}^{-2}) \varphi^*(x),\xi\otimes \eta\rangle-\langle(L_{\alpha(x)}^{-2}\otimes \alpha)\varphi^*(y),\xi\otimes \eta  \rangle\\
\nonumber&&+\langle(L_{\alpha(y)}^{-2})\otimes\alpha)\varphi^* x, \xi\otimes \eta\rangle-\langle (\alpha\otimes\ad_{\alpha(x)}^{-2})\varphi^*(y),\xi\otimes \eta\rangle\\
&=& \langle \varphi^\ast[x,y]-(L^{-2}_{\alpha(x)}\otimes\alpha+\alpha\otimes \ad^{-2}_{\alpha(x)})\varphi^\ast(y)+(L^{-2}_{\alpha(y)}\otimes\alpha+\alpha\otimes \ad^{-2}_{\alpha(y)})\varphi^\ast(x),\xi\otimes\eta\rangle,
\end{eqnarray}
which means that \eqref{pre-matche-pair-2} holds if and only if $\varphi^\ast$ is a $1$-cocycle of the sub-adjacent Hom-Lie algebra $A^C$ associated to the representation $(A\otimes A,L^{-2}\otimes\alpha+\alpha \otimes \ad^{-2})$. By Proposition \ref{matched-pair-equivalent}, we can obtain that $\eqref{pre-matche-pair-2} \Longleftrightarrow \eqref{pre-matche-pair-3}$. Therefore, $\varphi^\ast$ is a $1$-cocycle of the sub-adjacent Hom-Lie algebra  $A^C$ associated to the representation $(A\otimes A,L^{-2}\otimes\alpha+\alpha \otimes \ad^{-2})$ if and only if $\eqref{pre-matche-pair-2}$ and $\eqref{pre-matche-pair-3}$ hold. Similarly, we can prove that $\psi^\ast$ is a 1-cocycle of the sub-adjacent Hom-Lie algebra $(A^\ast)^C$ associated to the representation $(A^*\otimes A^*,\huaL^{-2}\otimes(\alpha^{-1})^\ast+(\alpha^{-1})^\ast\otimes \mathfrak{ad}^{-2})$ if and only if $\eqref{pre-matche-pair-1}$ and $\eqref{pre-matche-pair-4}$ hold. Thus, we obtain that condition $\rm(i)$ is equivalent to condition $\rm(ii)$.

Next, we prove that $\rm(ii)$ is equivalent to $\rm(iii)$.
Let $(A,A^\ast,\ad^\star,-R^\star,\mathfrak{ad}^\star,-\huaR^\star)$ be a matched pair of Hom-pre-Lie algebras. By Theorem \ref{matched-pair-Hom-pre-Lie algebra}, $(A\oplus A^\ast,\diamond,\alpha\oplus (\alpha^{-1})^*)$ is a Hom-pre-Lie algebra, where $``\diamond$'' is given by \eqref{Hom-pre-lie standard-matched pair}. Let $\omega$ be the natural nondegenerate skew-symmetric bilinear form  on $A\oplus A^*$ given by \eqref{standard manin triple}. We only need to prove that $\omega$ satisfies the invariant conditions. For all $x,y,z\in A, \xi,\eta,\gamma\in A^\ast$, we have
\begin{eqnarray}\label{equivalent-3-1}
\nonumber \omega(\alpha(x)+(\alpha^{-1})^\ast(\xi),\alpha(y)+(\alpha^{-1})^\ast(\eta))
\nonumber&=&\langle (\alpha^{-1})^\ast(\xi),\alpha(y)\rangle-\langle (\alpha^{-1})^\ast(\eta),\alpha(x)\rangle\\
\nonumber&=&\langle\xi,y\rangle-\langle\eta,x\rangle\\
 &=&\omega(x+\xi,y+\eta),
\end{eqnarray}
which implies that \eqref{invariant-1} holds. Moreover, we have
\begin{eqnarray}\label{equivalent-3-2}
\nonumber && \omega((x+\xi)\diamond(y+\eta),\alpha(z)+(\alpha^{-1})^\ast(\gamma))\\
\nonumber &=&\omega(x\cdot y+\mathfrak{ad}^\star_\xi y-\huaR^\star_\eta x+\xi\circ \eta+\ad^\star_x \eta-R^\star_y \xi,\alpha(z)+(\alpha^{-1})^\ast(\gamma))\\
\nonumber &=&\langle \xi\circ \eta+\ad^\star_x \eta-R^\star_y \xi,\alpha(z)\rangle-\langle(\alpha^{-1})^\ast(\gamma), x\cdot y+\mathfrak{ad}^\star_\xi y-\huaR^\star_\eta x\rangle\\
\nonumber &=&\langle\alpha^\ast(\xi)\circ \alpha^\ast(\eta),z \rangle-\langle \eta,\alpha^{-1}(x)\cdot \alpha^{-1}(z)-\alpha^{-1}(z)\cdot \alpha^{-1}(x)\rangle+\langle \xi,\alpha^{-1}(z)\cdot \alpha^{-1}(y)\rangle\\
&&-\langle \gamma,\alpha^{-1}(x)\cdot \alpha^{-1}(y)\rangle+\langle y,\alpha^\ast(\xi)\circ \alpha^\ast(\gamma)\rangle-\alpha^\ast(\gamma)\circ \alpha^\ast(\xi)\rangle-\langle x,\alpha^\ast(\gamma)\circ \alpha^\ast(\eta)\rangle,
\end{eqnarray}
and
\begin{eqnarray}\label{equivalent-3-3}
\nonumber && -\omega(\alpha(y)+(\alpha^{-1})^\ast(\eta),[x+\xi,z+\gamma])\\
\nonumber &=&-\omega(\alpha(y)+(\alpha^{-1})^\ast(\eta),x\cdot z+\mathfrak{ad}^\star_\xi z-\huaR^\star_\gamma x+\xi\circ \gamma+\ad^\star_x \gamma-R^\star_z \xi)\\
\nonumber&&-\omega(\alpha(y)+(\alpha^{-1})^\ast(\eta),-z\cdot x-\mathfrak{ad}^\star_\gamma x+\huaR^\star_\xi z-\gamma\circ \xi-\ad^\star_z \xi+R^\star_x \gamma)\\
\nonumber &=&-\langle (\alpha^{-1})^\ast(\eta),x\cdot z+\mathfrak{ad}^\star_\xi z-\huaR^\star_\gamma x-z\cdot x-\mathfrak{ad}^\star_\gamma x+\huaR^\star_\xi z\rangle\\
\nonumber&&+\langle\xi\circ \gamma+\ad^\star_x \gamma-R^\star_z \xi-\gamma\circ \xi-\ad^\star_z \xi+R^\star_x \gamma,\alpha(y)\rangle\\
\nonumber &=&-\langle\eta,\alpha^{-1}(x)\cdot \alpha^{-1}(z) \rangle+\langle z,\alpha^*(\xi)\circ\alpha^*(\eta)-\alpha^*(\eta)\circ\alpha^*(\xi)\rangle-\langle x,\alpha^*(\eta)\circ\alpha^*(\gamma)\rangle\\
\nonumber&& +\langle\eta,\alpha^{-1}(z)\cdot \alpha^{-1}(x) \rangle-\langle x,\alpha^*(\gamma)\circ\alpha^*(\eta)-\alpha^*(\eta)\circ\alpha^*(\gamma)\rangle+\langle z,\alpha^*(\eta)\circ\alpha^*(\xi)\rangle \\
\nonumber && +\langle \alpha^\ast(\xi)\circ \alpha^\ast(\gamma),y\rangle-\langle \gamma,\alpha^{-1}(x)\cdot\alpha^{-1}(y)-\alpha^{-1}(y)\cdot\alpha^{-1}(x)\rangle+\langle \xi, \alpha^{-1}(y)\cdot \alpha^{-1}(z)\rangle \\
\nonumber&& -\langle \alpha^\ast(\gamma)\circ \alpha^\ast(\xi),y\rangle+\langle \xi,\alpha^{-1}(z)\cdot\alpha^{-1}(y)-\alpha^{-1}(y)\cdot\alpha^{-1}(z)\rangle-\langle \gamma,\alpha^{-1}(y)\cdot \alpha^{-1}(x)\rangle \\
\nonumber &=& \langle\alpha^\ast(\xi)\circ \alpha^\ast(\eta),z \rangle-\langle \eta,\alpha^{-1}(x)\cdot \alpha^{-1}(z)-\alpha^{-1}(z)\cdot \alpha^{-1}(x)\rangle+\langle \xi,\alpha^{-1}(z)\cdot \alpha^{-1}(y)\rangle\\
&&-\langle \gamma,\alpha^{-1}(x)\cdot \alpha^{-1}(y)\rangle+\langle y,\alpha^\ast(\xi)\circ \alpha^\ast(\gamma)\rangle-\alpha^\ast(\gamma)\circ \alpha^\ast(\xi)\rangle-\langle x,\alpha^\ast(\gamma)\circ \alpha^\ast(\eta)\rangle,
\end{eqnarray}
which implies that \eqref{invariant-2} holds. Thus, $(A\oplus A^\ast,A,A^\ast)$ is a standard Manin triple.
Conversely, if $(A\oplus A^\ast,A,A^\ast)$ is the standard Manin triple, then it is obvious that $(A,A^\ast,\ad^\star,-R^\star,\mathfrak{ad}^\star,-\huaR^\star)$ is a matched pair of Hom-pre-Lie algebras.
\end{proof}

\section{Coboundary Hom-pre-Lie bialgebras}
In this section, we study coboundary Hom-pre-Lie bialgebras and introduce the notion of a Hom-$\frks$-matrix, which gives rise to a Hom-pre-Lie bialgebra naturally.

\begin{defi}
A Hom-pre-Lie bialgebra $(A,A^\ast,\varphi^\ast,\psi^\ast)$ is called {\bf coboundary} if $\varphi^\ast$ is a $1$-coboundary of the sub-adjacent Hom-Lie algebra $A^C$ associated to the representation $(A\otimes A,L^{-2}\otimes\alpha+\alpha \otimes \ad^{-2})$, that is, there exists an $r\in A\otimes A$ such that
\begin{equation}\label{coboundary}
\varphi^\ast(x)=(L_x^{-2}\otimes\alpha+\alpha \otimes \ad_x^{-2})r,\quad \forall x \in A.
\end{equation}
\end{defi}
Let $(A,\cdot,\alpha)$ be a Hom-pre-Lie algebra, and $r\in A\otimes A$, suppose that $\varphi^\ast$ is a $1$-coboundary of the sub-adjacent Hom-Lie algebra  $A^C$ associated to the representation $(A\otimes A,L^{-2}\otimes\alpha+\alpha \otimes \ad^{-2})$. Then it is obvious that $\varphi^\ast$ is a $1$-cocycle of the sub-adjacent Hom-Lie algebra  $A^C$ associated to the representation $(A\otimes A,L^{-2}\otimes\alpha+\alpha \otimes \ad^{-2})$. \begin{pro}\label{pro:condition}
With the above notations,  $(A,A^\ast,\varphi^*,\psi^*)$ is a Hom-pre-Lie bialgebra if and only if the following two conditions are satisfied:
\begin{itemize}
\item [$\rm(i)$] $\circ:A^\ast\otimes A^\ast\longrightarrow A^\ast$ define a Hom-pre-Lie algebra structure on $A^\ast$, where $``\circ$'' is given by $\langle \varphi^\ast(x),\xi\otimes\eta\rangle=\langle x,\xi\circ \eta\rangle$.
\item[$\rm(ii)$] $\psi^\ast$ is a $1$-cocycle of the sub-adjacent Hom-Lie algebra $(A^\ast)^C$ associated to the representation $(A^*\otimes A^*,\huaL^{-2}\otimes(\alpha^{-1})^\ast+(\alpha^{-1})^\ast\otimes \mathfrak{ad}^{-2})$, where $\psi^\ast:A^\ast\longrightarrow A^\ast\otimes A^\ast$ is given by $\langle \psi^\ast(\xi),x\otimes y\rangle=\langle \xi,x\cdot y\rangle$.
\end{itemize}
\end{pro}
For all $r\in A\otimes A$, the linear map $r^\sharp:A^*\longrightarrow A$ is defined by
 \begin{equation}
\langle r^\sharp(\xi),\eta\rangle=\langle r,\xi\otimes \eta\rangle,\quad \forall \xi,\eta \in A^*.
\end{equation}
\begin{pro}\label{O-Operator}
Let $(A,\cdot,\alpha)$ be a Hom-pre-Lie algebra and $\varphi^*:A\longrightarrow A\otimes A$ defined by \eqref{coboundary}. If $r\in A\otimes A$ satisfies
\begin{equation}\label{pro-1}
r^\sharp\circ (\alpha^{-1})^*=\alpha\circ r^\sharp.
\end{equation}
Then, for all $\xi, \eta\in A^*$, we have
\begin{equation}\label{pro-2}
\xi\circ\eta=\ad_{r^\sharp(\xi)}^\star\eta-R_{{\sigma(r)}^\sharp(\eta)}^\star\xi=\ad_{\alpha({r^\sharp}(\xi))}^*((\alpha^{-2})^*(\eta))-R_{\alpha({{\sigma(r)}^\sharp(\eta)})}^*((\alpha^{-2})^*(\xi)),
\end{equation}
where $\sigma:A\otimes A\longrightarrow A\otimes A$ is the flip operator defined by $\sigma(x\otimes y)=y\otimes x$ for all $x,y\in A$. Furthermore, we have
\begin{equation}\label{pro-3}
r^\sharp(\alpha^*(\xi))\cdot r^\sharp(\alpha^*(\eta))-r^\sharp(\alpha^*(\xi\circ \eta))=[[r,r]](\xi,\eta),
\end{equation}
where $[[r,r]]=[r_{12},r_{23}]-r_{13}\cdot r_{12}+r_{13}\cdot r_{23}$.
\end{pro}
Before proving this result, let us explain the notations. Let $(A,\cdot,\alpha)$ be a Hom-pre-Lie algebra and $r=\sum_i x_i\otimes y_i \in A\otimes A$. Set
$$r_{12}=\sum_ix_i\otimes y_i\otimes \alpha, \quad r_{13}=\sum_ix_i\otimes \alpha \otimes y_i, \quad r_{23}=\sum_i \alpha\otimes x_i \otimes y_i.$$
\begin{proof}
Let $r=\sum_i x_i\otimes y_i$. Here the Einstein summation convention is used. By \eqref{coboundary} and \eqref{pro-1}, for all $z\in A,~\xi,~\eta\in A^*$, we have
\begin{eqnarray}
\nonumber &&\langle z,\xi\circ \eta\rangle  = \langle\varphi^*(z),\xi\otimes \eta\rangle \\
\nonumber&=&\langle (L_z^{-2}\otimes\alpha+\alpha \otimes \ad_z^{-2})(x_i\otimes y_i),\xi\otimes \eta\rangle\\
\nonumber &=& \langle \alpha^{-2}(z)\cdot x_i \otimes\alpha(y_i),\xi\otimes \eta\rangle+\langle\alpha(x_i)\otimes[\alpha^{-2}(z),y_i],\xi\otimes \eta\rangle\\
\nonumber&=& \langle\alpha^{-2}(z)\cdot x_i ,\xi\rangle\langle\alpha(y_i),\eta\rangle+\langle\alpha(x_i),\xi\rangle\langle[\alpha^{-2}(z),y_i],\eta\rangle\\
\nonumber &=& \langle \alpha^{-2}(z)\cdot \langle\alpha(y_i),\eta\rangle x_i ,\xi\rangle+\langle[\alpha^{-2}(z),\langle\alpha(x_i),\xi\rangle y_i],\eta\rangle\\
\nonumber&=& \langle \alpha^{-2}(z)\cdot {\sigma(r)}^\sharp(\alpha^*(\eta)) ,\xi\rangle+\langle[\alpha^{-2}(z),r^\sharp(\alpha^*(\xi))],\eta\rangle\\
\nonumber &=& \langle R_{{\sigma(r)}^\sharp(\alpha^*(\eta))}\alpha^{-2}(z),\xi\rangle-\langle \ad_{r^\sharp(\alpha^*(\xi))}\alpha^{-2}(z),\eta\rangle \\
\nonumber&=& -\langle z,(\alpha^{-2})^* R_{{\sigma(r)}^\sharp(\alpha^*(\eta))}^*\xi\rangle+\langle z,(\alpha^{-2})^*\ad_{r^\sharp(\alpha^*(\xi))}^*\eta\rangle\\
\nonumber &=& -\langle z,(\alpha^{-2})^* R_{\alpha^{-1}({\sigma(r)}^\sharp(\alpha^*(\eta)))}^\star(\alpha^2)^*(\xi)\rangle+\langle z,(\alpha^{-2})^*\ad_{\alpha^{-1}(r^\sharp(\alpha^*(\xi)))}^\star(\alpha^2)^*(\eta)\rangle\\
\nonumber&=& -\langle z, R_{\alpha({\sigma(r)}^\sharp(\alpha^*(\eta)))}^\star\xi\rangle+\langle z,\ad_{\alpha(r^\sharp(\alpha^*(\xi)))}^\star\eta\rangle\\
\nonumber&=&-\langle z, R_{{\sigma(r)}^\sharp(\eta)}^\star\xi\rangle+\langle z,\ad_{r^\sharp(\xi)}^\star\eta\rangle,
\end{eqnarray}
which implies that \eqref{pro-2} holds.

Moreover, for all $\theta\in A^*$, we have
\begin{eqnarray}
\nonumber &&\langle\alpha(r^\sharp(\xi))\cdot \alpha(r^\sharp(\eta))-\alpha(r^\sharp(\xi\circ \eta)),\theta\rangle \\
\nonumber &=& \langle r^\sharp((\alpha^{-1})^*(\xi))\cdot r^\sharp((\alpha^{-1})^*(\eta)),\theta\rangle-\langle \xi\circ \eta,{\sigma(r)}^\sharp(\alpha^*(\theta))\rangle\\
\nonumber&=& \langle r^\sharp((\alpha^{-1})^*(\xi))\cdot r^\sharp((\alpha^{-1})^*(\eta)),\theta\rangle-\langle \ad_{\alpha({r^\sharp}(\xi))}^*((\alpha^{-2})^*(\eta))-R_{\alpha({{\sigma(r)}^\sharp(\eta)})}^*((\alpha^{-2})^*(\xi)),{\sigma(r)}^\sharp(\alpha^*(\theta))\rangle\\
\nonumber &=& \langle r^\sharp((\alpha^{-1})^*(\xi))\cdot r^\sharp((\alpha^{-1})^*(\eta)),\theta\rangle+\langle (\alpha^{-2})^*(\eta),r^\sharp((\alpha^{-1})^*(\xi))\cdot {\sigma(r)}^\sharp(\alpha^*(\theta))\rangle\\
\nonumber&&-\langle(\alpha^{-2})^*(\eta),{\sigma(r)}^\sharp(\alpha^*(\theta))\cdot  r^\sharp((\alpha^{-1})^*(\xi)) \rangle-\langle(\alpha^{-2})^*(\xi),{\sigma(r)}^\sharp(\alpha^*(\theta))\cdot {\sigma(r)}^\sharp((\alpha^{-1})^*(\eta))\rangle\\
\nonumber&=& \langle \langle x_i,(\alpha^{-1})^*(\xi)\rangle y_i\cdot \langle x_j,(\alpha^{-1})^*(\eta)\rangle y_j,\theta\rangle+\langle \langle x_i,(\alpha^{-1})^*(\xi)\rangle y_i\cdot \langle y_j,\alpha^*(\theta)\rangle x_j,(\alpha^{-2})^*(\eta)\rangle\\
\nonumber && -\langle \langle y_i,\alpha^*(\theta)\rangle x_i\cdot \langle x_j,(\alpha^{-1})^*(\xi)\rangle y_j,(\alpha^{-2})^*(\eta)\rangle-\langle \langle y_i,\alpha^*(\theta)\rangle x_i\cdot \langle y_j,(\alpha^{-1})^*(\eta)\rangle x_j,(\alpha^{-2})^*(\xi)\rangle\\
\nonumber&=& \langle x_i,(\alpha^{-1})^*(\xi)\rangle\langle x_j,(\alpha^{-1})^*(\eta)\rangle\langle y_i\cdot y_j,\theta\rangle+\langle x_i,(\alpha^{-1})^*(\xi)\rangle\langle y_j,\alpha^*(\theta)\rangle\langle y_i\cdot x_j,(\alpha^{-2})^*(\eta)\rangle\\
\nonumber && -\langle y_i,\alpha^*(\theta)\rangle\langle x_j,(\alpha^{-1})^*(\xi)\rangle\langle x_i\cdot y_j,(\alpha^{-2})^*(\eta)\rangle-\langle y_i,\alpha^*(\theta)\rangle\langle y_j,(\alpha^{-1})^*(\eta)\rangle\langle x_i\cdot x_j,(\alpha^{-2})^*(\xi)\rangle\\
\nonumber&=& \langle \alpha(x_i)\otimes \alpha(x_j)\otimes y_i\cdot y_j-\alpha(x_i)\otimes[x_j,y_i]\otimes \alpha(y_j)-x_i\cdot x_j\otimes \alpha(y_j)\otimes\alpha(y_i),(\alpha^{-2})^*(\xi)\otimes(\alpha^{-2})^*(\eta)\otimes\theta\rangle \\
\nonumber &=& [[r,r]]((\alpha^{-2})^*(\xi),(\alpha^{-2})^*(\eta),\theta) \\
&=&\nonumber\langle  [[r,r]]((\alpha^{-2})^*(\xi),(\alpha^{-2})^*(\eta)),\theta\rangle,
\end{eqnarray}
which implies that \eqref{pro-3} holds. This finishes the proof.
\end{proof}

\begin{cor}\label{cor:dualalg}
Let $(A,\cdot,\alpha)$ be a Hom-pre-Lie algebra. Let $r\in Sym^2(A)$ satisfying \eqref{pro-1} and $[[r,r]]=0$,
then $(A^\ast,\circ,(\alpha^{-1})^\ast)$ is a Hom-pre-Lie algebra, where $\circ$ is given by  \eqref{pro-2}.
\end{cor}

\begin{proof}
By Proposition \ref{O-Operator} and $(A^*,(\alpha^{-1})^*,\ad^\star,-R^\star)$  a representation of a Hom-pre-Lie algebra $(A,\cdot,\alpha)$. For all $\xi,\eta,\delta\in A^*,$ we have
\begin{eqnarray*}
 &&(\xi\circ \eta)\circ (\alpha^{-1})^\ast(\delta)-(\alpha^{-1})^\ast(\xi)\circ(\eta\circ \delta)-(\eta\circ \xi)\circ (\alpha^{-1})^\ast(\delta)+(\alpha^{-1})^\ast(\eta)\circ(\xi\circ \delta)\\
 &=&\ad_{r^\sharp(\xi\circ\eta)}^\star(\alpha^{-1})^\ast(\delta)-R_{r^\sharp((\alpha^{-1})^\ast(\delta))}^\star (\xi\circ\eta)-\ad_{r^\sharp((\alpha^{-1})^\ast(\xi))}^\star (\eta\circ\delta)+R_{r^\sharp(\eta\circ\delta)}^\star(\alpha^{-1})^\ast(\xi)\\
 &&-\ad_{r^\sharp(\eta\circ\xi)}^\star(\alpha^{-1})^\ast(\delta)+R_{r^\sharp((\alpha^{-1})^\ast(\delta))}^\star (\eta\circ \xi)+\ad_{r^\sharp((\alpha^{-1})^\ast(\eta))}^\star (\xi\circ\delta)-R_{r^\sharp(\xi\circ\delta)}^\star(\alpha^{-1})^\ast(\eta)\\
 &=&\ad_{r^\sharp(\xi\circ\eta)}^\star(\alpha^{-1})^\ast(\delta)-R_{r^\sharp((\alpha^{-1})^\ast(\delta))}^\star (\ad_{r^\sharp(\xi)}^\star\eta-R_{r^\sharp(\eta)}^\star\xi)\\
 &&-\ad_{r^\sharp((\alpha^{-1})^\ast(\xi))}^\star (\ad_{r^\sharp(\eta)}^\star\delta-R_{r^\sharp(\delta)}^\star\eta)+R_{r^\sharp(\eta\circ\delta)}^\star(\alpha^{-1})^\ast(\xi)\\
 &&-\ad_{r^\sharp(\eta\circ\xi)}^\star(\alpha^{-1})^\ast(\delta)+R_{r^\sharp((\alpha^{-1})^\ast(\delta))}^\star (\ad_{r^\sharp(\eta)}^\star\xi-R_{r^\sharp(\xi)}^\star\eta)\\
 &&+\ad_{r^\sharp((\alpha^{-1})^\ast(\eta))}^\star (\ad_{r^\sharp(\xi)}^\star\delta-R_{r^\sharp(\delta)}^\star\xi)-R_{r^\sharp(\xi\circ\delta)}^\star(\alpha^{-1})^\ast(\eta)\\
 &=&\ad_{r^\sharp(\xi\circ\eta)}^\star(\alpha^{-1})^\ast(\delta)-\ad_{r^\sharp(\eta\circ\xi)}^\star(\alpha^{-1})^\ast(\delta)-\ad_{\alpha(r^\sharp(\xi))}^\star\ad_{r^\sharp(\eta)}^\star\delta+\ad_{\alpha(r^\sharp(\eta))}^\star\ad_{r^\sharp(\xi)}^\star\delta\\
 &&-R_{\alpha(r^\sharp(\delta))}^\star L_{r^\sharp(\xi)}^\star \eta+L_{\alpha(r^\sharp(\xi))}^\star R_{r^\sharp(\delta)}^\star \eta-R_{\alpha(r^\sharp(\xi))}^\star R_{r^\sharp(\delta)}^\star \eta-R_{r^\sharp(\xi\circ\delta)}^\star (\alpha^{-1})^*(\eta)\\
 &&+R_{\alpha(r^\sharp(\delta))}^\star L_{r^\sharp(\eta)}^\star \xi-L_{\alpha(r^\sharp(\eta))}^\star R_{r^\sharp(\delta)}^\star \xi+R_{\alpha(r^\sharp(\eta))}^\star R_{r^\sharp(\delta)}^\star \xi+R_{r^\sharp(\eta\circ\delta)}^\star (\alpha^{-1})^*(\xi)\\
 &=&(\ad_{r^\sharp(\xi\circ\eta)}^\star-\ad_{r^\sharp(\eta\circ\xi)}^\star-\ad_{r^\sharp(\xi)\cdot r^\sharp(\eta)-r^\sharp(\eta)\cdot r^\sharp(\xi)}^\star)((\alpha^{-1})^\ast(\delta))\\
 &&+(R_{r^\sharp(\xi)\cdot r^\sharp(\delta)}^\star-R_{r^\sharp(\xi\circ\delta)}^\star)((\alpha^{-1})^\ast(\eta))-(R_{r^\sharp(\eta)\cdot r^\sharp(\delta)}^\star-R_{r^\sharp(\eta\circ\delta)}^\star)((\alpha^{-1})^\ast(\xi)).
\end{eqnarray*}
For all $ x\in A$, we have
\begin{eqnarray*}
 &&\langle(\xi\circ \eta)\circ (\alpha^{-1})^\ast(\delta)-(\alpha^{-1})^\ast(\xi)\circ(\eta\circ \delta)-(\eta\circ \xi)\circ (\alpha^{-1})^\ast(\delta)+(\alpha^{-1})^\ast(\eta)\circ(\xi\circ \delta),x\rangle\\
 &=&-\langle \ad_{{r^\sharp(\xi)\cdot r^\sharp(\eta)}-{r^\sharp(\xi\circ\eta)}}^\star(\alpha^{-1})^\ast(\delta),x\rangle+\langle \ad_{{r^\sharp(\eta)\cdot r^\sharp(\xi)}-{r^\sharp(\eta\circ\xi)})}^\star(\alpha^{-1})^\ast(\delta),x\rangle\\
 &&+\langle R_{{r^\sharp(\xi)\cdot r^\sharp(\delta)}-{r^\sharp(\xi\circ\delta)})}^\star(\alpha^{-1})^\ast(\eta),x\rangle-\langle R_{{r^\sharp(\eta)\cdot r^\sharp(\delta)}-{r^\sharp(\eta\circ\delta)})}^\star(\alpha^{-1})^\ast(\xi),x\rangle\\
 &=&-\langle \ad_{\alpha({r^\sharp(\xi))\cdot \alpha(r^\sharp(\eta))}-{\alpha(r^\sharp(\xi\circ\eta))}}^*(\alpha^{-3})^\ast(\delta),x\rangle+\langle \ad_{\alpha({r^\sharp(\eta))\cdot \alpha(r^\sharp(\xi))}-{\alpha(r^\sharp(\eta\circ\xi))}}^*(\alpha^{-3})^\ast(\delta),x\rangle\\
 &&+\langle R_{\alpha({r^\sharp(\xi))\cdot \alpha(r^\sharp(\delta))}-{\alpha(r^\sharp(\xi\circ\delta))}}^*(\alpha^{-3})^\ast(\eta),x\rangle-\langle R_{\alpha({r^\sharp(\eta))\cdot \alpha(r^\sharp(\delta))}-{\alpha(r^\sharp(\eta\circ\delta))}}^*(\alpha^{-3})^\ast(\xi),x\rangle\\
 &=&-\langle \ad_{[[r,r]]((\alpha^{-2})^*(\xi),(\alpha^{-2})^*(\eta))}^*(\alpha^{-3})^\ast(\delta),x\rangle+\langle \ad_{[[r,r]]((\alpha^{-2})^*(\eta),(\alpha^{-2})^*(\xi))}^*(\alpha^{-3})^\ast(\delta),x\rangle\\
 &&+\langle R_{[[r,r]]((\alpha^{-2})^*(\xi),(\alpha^{-2})^*(\delta))}^*(\alpha^{-3})^\ast(\eta),x\rangle-\langle R_{[[r,r]]((\alpha^{-2})^*(\eta),(\alpha^{-2})^*(\delta))}^*(\alpha^{-3})^\ast(\xi),x\rangle\\
 &=&0.
\end{eqnarray*}
Thus, $(A^\ast,\circ,(\alpha^{-1})^\ast)$ is a Hom-pre-Lie algebra.
\end{proof}

\begin{pro}\label{pro:cocycle}
Let $(A,\cdot,\alpha)$ be a Hom-pre-Lie algebra and $\varphi^*:A\longrightarrow A\otimes A$ defined by \eqref{coboundary}. If $r\in A\otimes A$ satisfies \eqref{pro-1} and $(A^*,\circ,(\alpha^{-1})^*)$ is a Hom-pre-Lie algebra, where $\circ$ is given by \eqref{pro-2}. Then $\psi^\ast$ is a $1$-cocycle of the sub-adjacent Hom-Lie algebra $(A^\ast)^C$ associated to the representation $(A^*\otimes A^*,\huaL^{-2}\otimes(\alpha^{-1})^\ast+(\alpha^{-1})^\ast\otimes \mathfrak{ad}^{-2})$ if and only if the following equation holds:
\begin{equation}\label{cocycle}
(P(x\cdot y)-P(\alpha(x))P(y))(r-\sigma(r))=0,
\end{equation}
where $P(x)=L^{-2}_x\otimes \alpha+\alpha\otimes L^{-2}_x$.
\end{pro}
\begin{proof}
By   \eqref{pro-1}, for all $\xi,\eta\in A^*$, we have
\begin{eqnarray}
\nonumber \langle r^\sharp((\alpha^{-1})^*(\xi))-\alpha(r^\sharp(\xi)),\eta\rangle&=&\langle r,(\alpha^{-1})^*(\xi)\otimes \eta\rangle-\langle r^\sharp(\xi),\alpha^*(\eta)\rangle \\
\nonumber&=&\langle (\alpha^{-1}\otimes {\Id})r,\xi\otimes \eta\rangle-\langle r,\xi\otimes \alpha^*(\eta)\rangle\\
\nonumber&=&\langle (\alpha^{-1}\otimes {\Id}-{\Id} \otimes \alpha)r,\xi\otimes \eta\rangle\\
\nonumber&=&0,
\end{eqnarray}
which implies that
\begin{equation}\label{1-cocycle-1}
(\alpha^{-1}\otimes {\Id})r=({\Id} \otimes \alpha)r.
\end{equation}
Let $r=\sum_i x_i\otimes y_i$. Here the Einstein summation convention is used. For all $x,y\in A, \xi,\eta\in A^*$, we have
\begin{eqnarray*}
&&\langle-L^\star_{\alpha(x)}\{\xi,\eta\}+\{L_x^\star \xi,(\alpha^{-1})^*(\eta)\}+\{(\alpha^{-1})^*(\xi),L_x^\star \eta\}+L^\star_{\huaL^\star_\eta x}(\alpha^{-1})^*(\xi)-L^\star_{\huaL^\star_\xi x}(\alpha^{-1})^*(\eta),\alpha^2(y)\rangle \\
&=&\langle \{\xi,\eta\},x\cdot y\rangle+\langle L^\star_x\xi\circ (\alpha^{-1})^*(\eta),\alpha^2(y) \rangle-\langle(\alpha^{-1})^*(\eta)\circ  L^\star_x\xi,\alpha^2(y)\rangle+\langle(\alpha^{-1})^*(\xi)\circ  L^\star_x\eta,\alpha^2(y)\rangle\\
&&- \langle L^\star_x\eta\circ (\alpha^{-1})^*(\xi),\alpha^2(y)\rangle-\langle(\alpha^{-1})^*(\xi),\alpha^{-1}(\huaL^\star_\eta x)\cdot y)\rangle +\langle(\alpha^{-1})^*(\eta),\alpha^{-1}(\huaL^\star_\xi x)\cdot y)\rangle\\
&=& \langle\xi\circ \eta,x\cdot y\rangle-\langle\eta\circ \xi,x\cdot y\rangle+\langle (\alpha^2)^*(L^\star_x\xi)\circ \alpha^*(\eta),y\rangle-\langle\alpha^*(\eta)\circ (\alpha^2)^*(L^\star_x\xi),y \rangle \\
&&+\langle\alpha^*(\xi)\circ (\alpha^2)^*(L^\star_x\eta),y\rangle -\langle (\alpha^2)^*(L^\star_x\eta)\circ \alpha^*(\xi),y\rangle-\langle(\alpha^{-1})^*(\xi),R_y \alpha^{-1}(\huaL^\star_\eta x) \rangle\\
&&+\langle(\alpha^{-1})^*(\eta),R_y \alpha^{-1}(\huaL^\star_\xi x) \rangle\\
&=& \langle\xi\otimes\eta,\varphi^*(x\cdot y)\rangle-\langle\eta\otimes \xi,\varphi^*(x\cdot y)\rangle+\langle L_{\alpha^{-2}(x)}^\star(\alpha^2)^*(\xi)\circ \alpha^*(\eta),y\rangle\\
&& -\langle  \alpha^*(\eta)\circ L_{\alpha^{-2}(x)}^\star(\alpha^2)^*(\xi),y \rangle+ \langle  \alpha^*(\xi)\circ L_{\alpha^{-2}(x)}^\star(\alpha^2)^*(\eta),y \rangle-\langle L_{\alpha^{-2}(x)}^\star(\alpha^2)^*(\eta)\circ \alpha^*(\xi),y\rangle\\
&&+\langle (\alpha^{-1})^*R_y^*(\alpha^{-1})^*(\xi),\huaL^\star_\eta x\rangle-\langle (\alpha^{-1})^*R_y^*(\alpha^{-1})^*(\eta),\huaL^\star_\xi x\rangle   \\
&=& \langle\xi\otimes\eta,\varphi^*(x\cdot y)\rangle-\langle\eta\otimes \xi,\varphi^*(x\cdot y)\rangle+\langle L_{\alpha^{-1}(x)}^*\xi\circ \alpha^*(\eta),y\rangle-\langle \alpha^*(\eta)\circ L_{\alpha^{-1}(x)}^*\xi,y\rangle\\
&& +\langle \alpha^*(\xi)\circ L_{\alpha^{-1}(x)}^*\eta,y\rangle-\langle L_{\alpha^{-1}(x)}^*\eta\circ \alpha^*(\xi),y\rangle+\langle (\alpha^{-1})^*R_{\alpha^{-1}(y)}^\star\alpha^*(\xi),\huaL^\star_\eta x \rangle\\
&&-\langle (\alpha^{-1})^*R_{\alpha^{-1}(y)}^\star\alpha^*(\eta),\huaL^\star_\xi x \rangle\\
&=&  \langle\xi\otimes\eta,\varphi^*(x\cdot y)\rangle-\langle\eta\otimes \xi,\varphi^*(x\cdot y)\rangle+\langle (L^{-2})^*_{\alpha(x)}\xi\circ \alpha^*(\eta),y\rangle-\langle \alpha^*(\eta)\circ (L^{-2})^*_{\alpha(x)}\xi,y\rangle\\
&& +\langle \alpha^*(\xi)\circ (L^{-2})^*_{\alpha(x)}\eta,y\rangle-\langle (L^{-2})^*_{\alpha(x)}\eta\circ \alpha^*(\xi),y\rangle+\langle R^\star_y\xi,\huaL^\star_\eta x \rangle-\langle R^\star_y\eta,\huaL^\star_\xi x \rangle\\
&=& \langle\xi\otimes\eta,\varphi^*(x\cdot y)\rangle-\langle\eta\otimes \xi,\varphi^*(x\cdot y)\rangle+\langle (L^{-2})^*_{\alpha(x)}\xi\otimes \alpha^*(\eta),\varphi^*(y)\rangle-\langle \alpha^*(\eta)\otimes(L^{-2})^*_{\alpha(x)}\xi,\varphi^*(y)\rangle\\
&& +\langle \alpha^*(\xi)\otimes (L^{-2})^*_{\alpha(x)}\eta,\varphi^*(y)\rangle-\langle (L^{-2})^*_{\alpha(x)}\eta\otimes \alpha^*(\xi),\varphi^*(y)\rangle-\langle \alpha^*(\eta)\circ (\alpha^2)^*(R^\star_y\xi), x \rangle\\
&&+\langle \alpha^*(\xi)\circ (\alpha^2)^*(R^\star_y\eta), x \rangle\\
&=& \langle\xi\otimes\eta,\varphi^*(x\cdot y)\rangle-\langle\eta\otimes \xi,\varphi^*(x\cdot y)\rangle+\langle (L^{-2})^*_{\alpha(x)}\xi\otimes \alpha^*(\eta),\varphi^*(y)\rangle-\langle \alpha^*(\eta)\otimes(L^{-2})^*_{\alpha(x)}\xi,\varphi^*(y)\rangle\\
&& +\langle \alpha^*(\xi)\otimes (L^{-2})^*_{\alpha(x)}\eta,\varphi^*(y)\rangle-\langle (L^{-2})^*_{\alpha(x)}\eta\otimes \alpha^*(\xi),\varphi^*(y)\rangle-\langle \alpha^*(\eta)\circ R^\star_{\alpha^{-2}(y)}(\alpha^2)^*(\xi), x \rangle\\
&&+\langle \alpha^*(\xi)\circ R^\star_{\alpha^{-2}(y)}(\alpha^2)^*(\eta), x \rangle\\
&=& \langle\xi\otimes\eta,\varphi^*(x\cdot y)\rangle-\langle\eta\otimes \xi,\varphi^*(x\cdot y)\rangle+\langle (L^{-2})^*_{\alpha(x)}\xi\otimes \alpha^*(\eta),\varphi^*(y)\rangle-\langle \alpha^*(\eta)\otimes(L^{-2})^*_{\alpha(x)}\xi,\varphi^*(y)\rangle\\
&& +\langle \alpha^*(\xi)\otimes (L^{-2})^*_{\alpha(x)}\eta,\varphi^*(y)\rangle-\langle (L^{-2})^*_{\alpha(x)}\eta\otimes \alpha^*(\xi),\varphi^*(y)\rangle-\langle \alpha^*(\eta)\circ R^*_{\alpha^{-1}(y)}\xi, x \rangle\\
&&+\langle \alpha^*(\xi)\circ R^*_{\alpha^{-1}(y)}\eta, x \rangle\\
&=&\langle\xi\otimes\eta,\varphi^*(x\cdot y)\rangle-\langle\eta\otimes \xi,\varphi^*(x\cdot y)\rangle+\langle (L^{-2})^*_{\alpha(x)}\xi\otimes \alpha^*(\eta),\varphi^*(y)\rangle-\langle \alpha^*(\eta)\otimes(L^{-2})^*_{\alpha(x)}\xi,\varphi^*(y)\rangle\\
&& +\langle \alpha^*(\xi)\otimes (L^{-2})^*_{\alpha(x)}\eta,\varphi^*(y)\rangle-\langle (L^{-2})^*_{\alpha(x)}\eta\otimes \alpha^*(\xi),\varphi^*(y)\rangle-\langle \alpha^*(\eta)\circ (R^{-2})^*_{\alpha(y)}\xi, x \rangle\\
&&+\langle \alpha^*(\xi)\circ (R^{-2})^*_{\alpha(y)}\eta, x \rangle\\
&=&\langle\xi\otimes\eta,\varphi^*(x\cdot y)\rangle-\langle\xi\otimes\eta,(L^{-2}_{\alpha(x)}\otimes \alpha)\varphi^*(y) \rangle-\langle\xi\otimes\eta,(\alpha \otimes L^{-2}_{\alpha(x)})\varphi^*(y)\rangle-\langle\xi\otimes\eta,(\alpha \otimes R^{-2}_{\alpha(y)})\varphi^*(x)\rangle\\
&&-\langle\eta\otimes\xi,\varphi^*(x\cdot y)\rangle+\langle\eta\otimes\xi,(L^{-2}_{\alpha(x)}\otimes \alpha)\varphi^*(y) \rangle+\langle\eta\otimes\xi,(\alpha \otimes L^{-2}_{\alpha(x)})\varphi^*(y)\rangle+\langle\eta\otimes\xi,(\alpha \otimes R^{-2}_{\alpha(y)})\varphi^*(x)\rangle
\end{eqnarray*}
Moreover, we have
\begin{eqnarray}\label{1-cocycle-2}
\nonumber \langle\eta\otimes\xi,\varphi^*(x\cdot y)\rangle
\nonumber &=& \langle \eta\otimes\xi,(L^{-2}_{x\cdot y}\otimes \alpha+\alpha\otimes \ad^{-2}_{x\cdot y})r\rangle \\
\nonumber &=& \langle \eta\otimes\xi,(L^{-2}_{x\cdot y}\otimes \alpha+\alpha\otimes \ad^{-2}_{x\cdot y})(x_i \otimes y_i)\rangle\\
\nonumber &=&\langle \eta\otimes\xi,L^{-2}_{x\cdot y} x_i\otimes \alpha(y_i)+\alpha(x_i)\otimes \ad^{-2}_{x\cdot y}y_i\rangle\\
\nonumber &=&\langle \xi\otimes \eta,\alpha(y_i)\otimes L^{-2}_{x\cdot y} x_i+\ad^{-2}_{x\cdot y}y_i \otimes \alpha(x_i)\rangle\\
&=&\langle \xi\otimes \eta,(\alpha \otimes L^{-2}_{x\cdot y}+\ad^{-2}_{x\cdot y}\otimes \alpha)\sigma(r)\rangle.
\end{eqnarray}
Similarly, we have
\begin{eqnarray}
\label{1-cocycle-3}\langle\eta\otimes\xi,(L^{-2}_{\alpha(x)}\otimes \alpha)\varphi^*(y) \rangle=\langle\xi\otimes \eta,(\alpha \otimes L^{-2}_{\alpha(x)})(\alpha\otimes L^{-2}_y+\ad^{-2}_y \otimes \alpha)\sigma(r) \rangle,\\
\label{1-cocycle-4}\langle\eta\otimes\xi,(\alpha\otimes L^{-2}_{\alpha(x)})\varphi^*(y) \rangle=\langle\xi\otimes \eta,( L^{-2}_{\alpha(x)}\otimes \alpha)(\alpha\otimes L^{-2}_y+\ad^{-2}_y \otimes \alpha)\sigma(r) \rangle,\\
\label{1-cocycle-5}\langle\eta\otimes\xi,(\alpha\otimes R^{-2}_{\alpha(y)})\varphi^*(x) \rangle=\langle\xi\otimes \eta,( R^{-2}_{\alpha(y)}\otimes \alpha)(\alpha\otimes L^{-2}_x+\ad^{-2}_x \otimes \alpha)\sigma(r) \rangle.
\end{eqnarray}
Thus, by \eqref{1-cocycle-1}-\eqref{1-cocycle-5} and the definition of a Hom-pre-Lie algebra, we have
\begin{eqnarray*}
&&\langle-L^\star_{\alpha(x)}\{\xi,\eta\}+\{L_x^\star \xi,(\alpha^{-1})^*(\eta)\}+\{(\alpha^{-1})^*(\xi),L_x^\star \eta\}+L^\star_{\huaL^\star_\eta x}(\alpha^{-1})^*(\xi)-L^\star_{\huaL^\star_\xi x}(\alpha^{-1})^*(\eta),\alpha^2(y)\rangle \\
&=&\langle\xi\otimes\eta,(L^{-2}_{x\cdot y}\otimes \alpha+\alpha\otimes \ad^{-2}_{x\cdot y})r\rangle-\langle\xi\otimes\eta,(L^{-2}_{\alpha(x)}\otimes \alpha)(L^{-2}_ y\otimes \alpha+\alpha\otimes \ad^{-2}_y)r \rangle\\
&&-\langle\xi\otimes\eta,(\alpha \otimes L^{-2}_{\alpha(x)})(L^{-2}_ y\otimes \alpha+\alpha\otimes \ad^{-2}_y)r\rangle-\langle\xi\otimes\eta,(\alpha \otimes R^{-2}_{\alpha(y)})(L^{-2}_ x\otimes \alpha+\alpha\otimes \ad^{-2}_x)r\rangle\\
&&-\langle\xi\otimes\eta,(\alpha \otimes L^{-2}_{x\cdot y}+\ad^{-2}_{x\cdot y} \otimes\alpha)\sigma(r)\rangle+\langle\xi\otimes\eta,(\alpha \otimes L^{-2}_{\alpha(x)})(\alpha \otimes L^{-2}_ y+ \ad^{-2}_y \otimes \alpha)\sigma(r)\rangle\\
&&+\langle\xi\otimes\eta,(L^{-2}_{\alpha(x)}\otimes \alpha)(\alpha \otimes L^{-2}_ y+ \ad^{-2}_y \otimes \alpha)\sigma(r)\rangle+\langle\xi\otimes\eta,(R^{-2}_{\alpha(y)}\otimes \alpha)(\alpha \otimes L^{-2}_ x+ \ad^{-2}_x \otimes \alpha)\sigma(r)\rangle\\
&=&\langle \xi\otimes \eta,\alpha^{-2}(x\cdot y)\cdot x_i\otimes \alpha(y_i)\rangle+\alpha(x_i)\otimes \alpha^{-2}(x\cdot y)\cdot y_i-\alpha^{-1}(x)\cdot(\alpha^{-2}(y)\cdot x_i)\otimes\alpha^2(y_i)\\
&&-\alpha^{-1}(x)\cdot \alpha(x_i)\otimes \alpha^{-1}(y)\cdot \alpha(y_i)-\alpha^{-1}(y)\cdot \alpha(x_i)\otimes \alpha^{-1}(x)\cdot \alpha(y_i)-\alpha^2(x_i)\otimes \alpha^{-1}(x)\cdot(\alpha^{-2}(y)\cdot y_i)\rangle\\
&&+\langle \xi\otimes \eta, -\alpha(y_i)\otimes \alpha^{-2}(x\cdot y)\cdot x_i-\alpha^{-2}(x\cdot y)\cdot y_i\otimes\alpha(x_i)+\alpha^2(y_i)\otimes \alpha^{-1}(x)\cdot(\alpha^{-2}(y)\cdot x_i)\rangle\\
&&+\alpha^{-1}(y)\cdot \alpha(y_i)\otimes \alpha^{-1}(x)\cdot \alpha(x_i)+\alpha^{-1}(x)\cdot \alpha(y_i)\otimes\alpha^{-1}(y)\cdot \alpha(x_i)+\alpha^{-1}(x)\cdot(\alpha^{-2}(y)\cdot y_i)\otimes \alpha^2(x_i)\rangle\\
&=&\langle \xi\otimes \eta,(L^{-2}_{x\cdot y}\otimes \alpha+\alpha\otimes L^{-2}_{x\cdot y}-(L^{-2}_{\alpha(x)}\otimes \alpha+\alpha\otimes L^{-2}_{\alpha(x)})(L^{-2}_y\otimes \alpha+\alpha\otimes L^{-2}_y))(x_i\otimes y_i)\rangle\\
&&-\langle \xi\otimes \eta,(L^{-2}_{x\cdot y}\otimes \alpha+\alpha\otimes L^{-2}_{x\cdot y}-(L^{-2}_{\alpha(x)}\otimes \alpha+\alpha\otimes L^{-2}_{\alpha(x)})(L^{-2}_y\otimes \alpha+\alpha\otimes L^{-2}_y))(y_i\otimes x_i)\rangle\\
&=&\langle \xi\otimes \eta, (P(x\cdot y)-P(\alpha(x))P(y))(r-\sigma(r)) \rangle,
\end{eqnarray*}
which implies that  \eqref{cocycle} holds if and only if \eqref{matche-pair-2} holds. By Proposition \ref{matched-pair-equivalent}, $\eqref{matche-pair-2}\Longleftrightarrow\eqref{pre-matche-pair-1} \Longleftrightarrow \eqref{pre-matche-pair-4}$, and by Theorem \ref{equivalent}, $\psi^\ast$ is a 1-cocycle of the sub-adjacent Hom-Lie algebra $(A^\ast)^C$ associated to the representation $(A^*\otimes A^*,\huaL^{-2}\otimes(\alpha^{-1})^\ast+(\alpha^{-1})^\ast\otimes \mathfrak{ad}^{-2})$ if and only if $\eqref{pre-matche-pair-1}$ and $\eqref{pre-matche-pair-4}$ hold. Thus, $\psi^\ast$ is a $1$-cocycle of the sub-adjacent Hom-Lie algebra $(A^\ast)^C$ associated to the representation $(A^*\otimes A^*,\huaL^{-2}\otimes(\alpha^{-1})^\ast+(\alpha^{-1})^\ast\otimes \mathfrak{ad}^{-2})$ if and only if \eqref{cocycle} holds.
\end{proof}
\begin{defi}
Let $(A,\cdot,\alpha)$ be a Hom-pre-Lie algebra. Assume that $r\in A\otimes A$ is symmetric and satisfies \eqref{pro-1}. Then the equation $[[r,r]]=0$ is called the {\bf Hom-$\frks$-equation} in $(A,\cdot,\alpha)$ and $r$ is called a {\bf Hom-$\frks$-matrix}. A {\bf triangular Hom-pre-Lie bialgebra} is a coboundary Hom-pre-Lie bialgebra, in which $r$ is a Hom-$\frks$-matrix.
\end{defi}

\begin{thm}
  Let $(A,\cdot,\alpha)$ be a Hom-pre-Lie algebra and $r$   a   Hom-$\frks$-matrix. Then $(A,A^\ast,\varphi^*,\psi^*)$ is a Hom-pre-Lie bialgebra, where $\varphi^*:A\longrightarrow A\otimes A$ defined by \eqref{coboundary} and $\psi^*$ is the dual of the multiplication $\cdot$ in $A$.
\end{thm}
\begin{proof}
 By Corollary \ref{cor:dualalg}, $(A^\ast,\circ,(\alpha^{-1})^\ast)$ is a Hom-pre-Lie algebra, where $\circ$ is given by  \eqref{pro-2}. By Proposition \ref{pro:cocycle}, $\psi^\ast$ is a $1$-cocycle of the sub-adjacent Hom-Lie algebra $(A^\ast)^C$ associated to the representation $(A^*\otimes A^*,\huaL^{-2}\otimes(\alpha^{-1})^\ast+(\alpha^{-1})^\ast\otimes \mathfrak{ad}^{-2})$.  By Proposition \ref{pro:condition}  $(A,A^\ast,\varphi^*,\psi^*)$ is a Hom-pre-Lie bialgebra.
\end{proof}


\section{Hom-$\huaO$-operators, Hom-L-dendriform algebras and Hom-$\frks$-matrices}

In this section, we introduce the notions of Hom-$\huaO$-operator on Hom-pre-Lie algebra and Hom-L-dendriform algebra, by which we construct Hom-$\frks$-matrices.
\begin{defi}
Let $(A,\cdot,\alpha)$ be a Hom-pre-Lie algebra and $(V,\beta,\rho,\mu)$ be a representation of $(A,\cdot,\alpha)$. A linear map $T:V\longrightarrow A$ is called {\bf Hom-$\huaO$-operator} if for all $u,v\in V$, the following equalities are satisfied
\begin{eqnarray}
  \label{eq:opreator-1}T\circ \beta&=&\alpha\circ T,\\
   \label{eq:operator-2}T(u)\cdot T(v)&=&T(\rho(T(\beta^{-1}(u)))(v)+\mu(T(\beta^{-1}(v)))(u)).
\end{eqnarray}
\end{defi}
\begin{pro}
Let $(A,\cdot,\alpha)$ be a Hom-pre-Lie algebra and $r\in A\otimes A$ is symmetric. Then $r$ satisfies \eqref{pro-1} and $[[r,r]]=0$ if and only if $r^{\sharp}\circ (\alpha^{-1})^*$ is a Hom-$\huaO$-operator associated to the representation $(A^*,(\alpha^{-1})^*,\ad^\star,-R^\star)$.
\end{pro}
\begin{proof}
By Proposition \ref{O-Operator}, it is straightforward.
\end{proof}

\begin{defi}
A {\bf Hom-L-dendriform algebra} $(A,\triangleright,\triangleleft,\alpha)$ is a vector space $A$ equipped with two bilinear products $\triangleright,\triangleleft:A\otimes A\longrightarrow A$ and $\alpha\in \gl(V)$, such that for all $x,y,z \in A$, $\alpha(x\triangleright y)=\alpha(x)\triangleright \alpha(y)$, $\alpha(x\triangleleft y)=\alpha(x)\triangleleft \alpha(y)$, and the following equalities are satisfied
\begin{eqnarray}
  \label{eq:L-1}\\ \nonumber (x\triangleright y)\triangleright \alpha(z)+(x\triangleleft y)\triangleright \alpha(z)+\alpha(y)\triangleright(x\triangleright z)-(y\triangleleft x)\triangleright\alpha(z)-(y\triangleright x)\triangleright\alpha(z)-\alpha(x)\triangleright(y\triangleright z)=0,\\
   \label{eq:L-2}(x\triangleright y)\triangleleft\alpha(z)+\alpha(y)\triangleleft(x\triangleright z)+\alpha(y)\triangleleft(x\triangleleft z)-(y\triangleleft x)\triangleleft \alpha(z)-\alpha(x)\triangleright(y\triangleleft z)=0.
\end{eqnarray}
\end{defi}

\begin{pro}
Let $(A,\triangleright,\triangleleft,\alpha)$ be a Hom-L-dendriform algebra.
\begin{itemize}
\item [$\rm(1)$] The bilinear product $\bullet:A\otimes A\longrightarrow A$ given by
\begin{equation}
x\bullet y=x\triangleright y+x\triangleleft y, \quad \forall x,y \in A,
\end{equation}
define a Hom-pre-Lie algebra. $(A,\bullet,\alpha)$ is called the associated horizontal Hom-pre-Lie algebra of $(A,\triangleright,\triangleleft,\alpha)$ and $(A,\triangleright,\triangleleft,\alpha)$ is called a compatible Hom-L-dendriform algebra structure on the Hom-pre-Lie algebra $(A,\bullet,\alpha)$.
\item[$\rm(2)$] The bilinear product $\cdot:A\otimes A\longrightarrow A$ given by
\begin{equation}
x\cdot y=x\triangleright y-y\triangleleft x, \quad \forall x,y \in A,
\end{equation}
defines a Hom-pre-Lie algebra. $(A,\cdot,\alpha)$ is called the associated vertical Hom-pre-Lie algebra of $(A,\triangleright,\triangleleft,\alpha)$ and $(A,\triangleright,\triangleleft,\alpha)$ is called a compatible Hom-L-dendriform algebra structure on the Hom-pre-Lie algebra $(A,\cdot,\alpha)$.
\item[$\rm(3)$]Both $(A,\bullet,\alpha)$ and $(A,\cdot,\alpha)$ have the same sub-adjacent Hom-Lie algebra $A^C$ defined by
\begin{equation}
[x,y]=x\triangleright y+x\triangleleft y-y\triangleleft x-y\triangleright x, \quad \forall x,y \in A.
\end{equation}
\end{itemize}
\end{pro}
\begin{proof}
It is straightforward.
\end{proof}
\begin{rmk}
Let $(A,\triangleright,\triangleleft,\alpha)$ be a Hom-L-dendriform algebra. Then \eqref{eq:L-1} and \eqref{eq:L-2} can be rewritten as
\begin{eqnarray}
   \ \alpha(x)\triangleright(y\triangleright z)-(x\bullet y)\triangleright \alpha(z)&=&\alpha(y)\triangleright (x\triangleright z)-(y\bullet x)\triangleright \alpha(z),\\
   \ \alpha(x)\triangleright(y\triangleleft z)-(x\triangleright y)\triangleleft \alpha(z)&=&\alpha(y)\triangleleft (x\bullet z)-(y\triangleleft x)\triangleleft\alpha(z).
\end{eqnarray}
\end{rmk}

\begin{pro}\label{rep}
Let $A$ be a vector space with two bilinear products $\triangleright,\triangleleft:A\otimes A\longrightarrow A$.
\begin{itemize}
\item [$\rm(1)$] $(A,\triangleright,\triangleleft,\alpha)$ is a Hom-L-dendriform algebra if and only if $(A,\bullet,\alpha)$ is a Hom-pre-Lie algebra and  $(A,\alpha,L_\triangleright,R_\triangleleft)$ is a representation of $(A,\bullet,\alpha)$.
\item[$\rm(2)$]$(A,\triangleright,\triangleleft,\alpha)$ is a Hom-L-dendriform algebra if and only if $(A,\cdot,\alpha)$ is a Hom-pre-Lie algebra and $(A,\alpha,L_\triangleright,-L_\triangleleft)$ is a representation of $(A,\cdot,\alpha)$.
\end{itemize}
\end{pro}
\begin{proof}
We only prove the condition $\rm(1)$. If $(A,\triangleright,\triangleleft,\alpha)$ is a Hom-L-dendriform algebra, then for all $x,y\in A$, we have
\begin{equation}
L_\triangleright(\alpha(x))\alpha(y)=\alpha(x)\triangleright\alpha(y)=\alpha(x\triangleright y)=\alpha(L_\triangleright (x)y),
\end{equation}
which implies that $L_\triangleright(\alpha(x))\circ \alpha=\alpha\circ L_\triangleright(x)$.
Similarly, we have $R_\triangleleft(\alpha(x))\circ \alpha=\alpha\circ R_\triangleleft(x)$.

For all $x,y,z\in A$, by \eqref{eq:L-1}, we have
\begin{eqnarray*}&&L_\triangleright([x,y])\alpha(z)-L_\triangleright(\alpha(x))L_\triangleright(y)z+L_\triangleright(\alpha(y))L_\triangleright(x)z\\
&=& [x,y]\triangleright \alpha(z)-\alpha(x)\triangleright(y\triangleright z)+\alpha(y)\triangleright(x\triangleright z) \\
&=& (x\triangleright y)\triangleright\alpha(z)+(x\triangleleft y)\triangleright\alpha(z)-(y\triangleright x)\triangleright\alpha(z)-(y\triangleleft x)\triangleright\alpha(z)-\alpha(x)\triangleright(y\triangleright z)+\alpha(y)\triangleright(x\triangleright z) \\
&=&0,
\end{eqnarray*}
which implies that $$L_\triangleright([x,y])\circ\alpha=L_\triangleright(\alpha(x))L_\triangleright(y)-L_\triangleright(\alpha(y))L_\triangleright(x).$$ Similarly, we have $$R_\triangleleft(\alpha(y))\circ R_\triangleleft(x)-R_\triangleleft(x\bullet y)\circ\alpha=R_\triangleleft(\alpha(y))\circ L_\triangleright(x)-L_\triangleright(\alpha(x))\circ R_\triangleleft(y).$$
Thus $(A,\alpha,L_\triangleright,R_\triangleleft)$ is a representation of the Hom-pre-Lie algebra $(A,\bullet,\alpha)$.
The converse part can be proved similarly. We omit details. The proof is finished.
\end{proof}

\begin{pro}
Let $(A,\triangleright,\triangleleft,\alpha)$ be a Hom-L-dendriform algebra. Define two bilinear products $\triangleright^t,\triangleleft^t:A\otimes A\longrightarrow A$ by
\begin{equation}\label{transpose}
x\triangleright^t y=x\triangleright y, \quad x\triangleleft^ty=-y\triangleleft x,  \quad \forall x,y \in A.
\end{equation}
Then $(A,\triangleright^t,\triangleleft^t,\alpha)$ is a Hom-L-dendriform algebra. The associated horizontal Hom-pre-Lie algebra of $(A,\triangleright^t,\triangleleft^t,\alpha)$ is the associated vertical Hom-pre-Lie algebra $(A,\cdot,\alpha)$ of $(A,\triangleright,\triangleleft,\alpha)$ and  the associated vertical Hom-pre-Lie algebra of $(A,\triangleright^t,\triangleleft^t,\alpha)$ is the associated horizontal Hom-pre-Lie algebra $(A,\bullet,\alpha)$ of $(A,\triangleright,\triangleleft,\alpha)$, that is, $$\bullet^t=\cdot,\quad \cdot^t=\bullet.$$
\end{pro}
\begin{proof}
It is straightforward.
\end{proof}
\begin{defi}
 Let $(A,\triangleright,\triangleleft,\alpha)$ be a Hom-L-dendriform algebra. The Hom-L-dendriform algebra $(A,\triangleright^t,\triangleleft^t,\alpha)$ given by \eqref{transpose} is called the transpose of $(A,\triangleright,\triangleleft,\alpha)$.
\end{defi}
For brevity, we only give the study of the vertical Hom-pre-Lie algebra.
\begin{thm}\label{O-operator-L}
Let $(A,\cdot,\alpha)$ be a Hom-pre-Lie algebra and $(V,\beta,\rho,\mu)$ be a representation of $(A,\cdot,\alpha)$. Suppose that $T:V\longrightarrow A$ is a Hom-$\huaO$-operator. Then there exists a Hom-L-dendriform algebra structure on $V$ defined by
\begin{equation}
u\triangleright v=\rho(T(\beta^{-1}(u)))v, \quad u\triangleleft v=-\mu(T(\beta^{-1}(u)))v,  \quad \forall u,v \in V.
\end{equation}
\end{thm}
\begin{proof}
First by $T\circ \beta=\alpha\circ T$ and $(V,\beta,\rho,\mu)$ is a representation of $(A,\cdot,\alpha)$, for all $u,v\in V$ we have
\begin{equation}
\beta(u\triangleright v)=\beta(\rho(T(\beta^{-1}(u)))v)=\rho(\alpha(T(\beta^{-1}(u))))\beta(v)=\rho(T(u))\beta(v)=\beta(u)\triangleright\beta(v).
\end{equation}
Similarly, we have $\beta(u\triangleleft v)=\beta(u)\triangleleft\beta(v).$ Furthermore, by \eqref{eq:operator-2} and $(V,\beta,\rho,\mu)$ being  a representation of $(A,\cdot,\alpha)$, for all $u,v,w\in V$, we have
\begin{eqnarray*}&&(u\triangleright v)\triangleright\beta(w)+(u\triangleleft v)\triangleright\beta(w)+\beta(v)\triangleright(u\triangleright w)-(v\triangleleft u)\triangleright\beta(w)-(v\triangleright u)\triangleright\beta(w)-\beta(u)\triangleright(v\triangleright w)\\
&=& \rho(T(\beta^{-1}(u)))v\triangleright\beta(w)-\mu(T(\beta^{-1}(u)))v\triangleright\beta(w)+\beta(v)\triangleright \rho(T(\beta^{-1}(u)))w\\
&&+\mu(T(\beta^{-1}(v)))u\triangleright\beta(w)-\rho(T(\beta^{-1}(v)))u\triangleright\beta(w)-\beta(u)\triangleright\rho(T(\beta^{-1}(v)))w\\
&=& \rho(T(\beta^{-1}(\rho(T(\beta^{-1}(u)))v)))\beta(w)-\rho(T(\beta^{-1}(\mu(T(\beta^{-1}(u)))v)))\beta(w)+\rho(T(v))\rho(T(\beta^{-1}(u)))w \\
&&+\rho(T(\beta^{-1}(\mu(T(\beta^{-1}(v)))u)))\beta(w)-\rho(T(\beta^{-1}(\rho(T(\beta^{-1}(v)))u)))\beta(w)-\rho(T(u))\rho(T(\beta^{-1}(v)))w\\
&=&\rho(\alpha^{-1}(T(u)\cdot T(v)))\beta(w)-\rho(\alpha^{-1}(T(v)\cdot T(u)))\beta(w)+\rho(T(v))\rho(T(\beta^{-1}(u)))w-\rho(T(u))\rho(T(\beta^{-1}(v)))w\\
&=&0,
\end{eqnarray*}
which implies that \eqref{eq:L-1} holds.

Similarly, we have
\begin{equation}
(u\triangleright v)\triangleleft\beta(w)+\beta(v)\triangleleft(u\triangleright w)+\beta(v)\triangleleft(u\triangleleft w)-(v\triangleleft u)\triangleleft\beta(w)-\beta(u)\triangleright(v\triangleleft w)=0,
\end{equation}
which implies that \eqref{eq:L-2} holds. This finishes the proof.
\end{proof}

\begin{cor}\label{O-operator-L1}
With the above conditions. $T$ is a homomorphism  from the associated vertical Hom-pre-Lie algebra of $(V,\triangleright,\triangleleft,\beta)$ to Hom-pre-Lie algebra $(A,\cdot,\alpha)$. Moreover, $T(V)=\{T(u)|u\in V\}\subset A$ is a Hom-pre-Lie subalgebra of $(A,\cdot,\alpha)$ and there is an induced Hom-L-dendriform algebra structure on $T(V)$ given by
\begin{equation}
T(u)\triangleright T(v)=T(u\triangleright v), \quad T(u)\triangleleft T(v)=T(u\triangleleft v),  \quad \forall u,v \in V.
\end{equation}
\end{cor}
\begin{thm}\label{Operator-L}
Let $(A,\cdot,\alpha)$ be a Hom-pre-Lie algebra. Then there exists a compatible Hom-L-dendriform algebra structure on $(A,\cdot,\alpha)$ such that $(A,\cdot,\alpha)$ is the associated vertical Hom-pre-Lie algebra if and only if there exists an invertible Hom-$\huaO$-operator $T$ associated to a representation $(V,\beta,\rho,\mu)$.
\end{thm}
\begin{proof}
Let $T$ be an invertible Hom-$\huaO$-operator $T$ associated to a representation $(V,\beta,\rho,\mu)$. By Theorem \ref{O-operator-L}, Corollary \ref{O-operator-L1} and \eqref{eq:opreator-1}, there exists a Hom-L-dendriform algebra on $T(V)$ given by
\begin{equation}
x\triangleright y=T(\rho(T(\beta^{-1}(u)))T^{-1}(y))=T(\rho(\alpha^{-1}(x))T^{-1}(y)).
\end{equation}
Similarly, we have $x\triangleleft y=-T(\mu(\alpha^{-1}(x))T^{-1}(y)).$

Moreover, by \eqref{eq:operator-2}, we have
\begin{equation}
x\triangleright y-y\triangleleft x=T(\rho(\alpha^{-1}(x))T^{-1}(y)+\mu(\alpha^{-1}(y))T^{-1}(x))=x\cdot y.
\end{equation}
Conversely, Let $(A,\triangleright,\triangleleft,\alpha)$ be a Hom-L-dendriform algebra and $(A,\cdot,\alpha)$ be the associated vertical Hom-pre-Lie algebra. By Theorem \ref{rep}, $(A,\alpha,L_\triangleright,-L_\triangleleft)$ is a representation of $(A,\cdot,\alpha)$ and $\alpha:A\longrightarrow A$ is a Hom-$\huaO$-operator of $(A,\cdot,\alpha)$ associated to $(A,\alpha,L_\triangleright,-L_\triangleleft)$.
\end{proof}

In the sequel, we give the relation between  Hessian structure and Hom-L-dendriform algebras.
\begin{defi}{\rm(\cite{LSS})}
A Hessian structure on a regular Hom-pre-Lie algebra $(A,\cdot,\alpha)$ is a symmetric nondegenerate $2$-cocycle $\huaB \in Sym^2(A^\ast)$, i.e. $\partial_T\huaB=0$, satisfying $\huaB \circ (\alpha\otimes\alpha)=\huaB$.  More precisely,
\begin{eqnarray}
  \label{hom-hessian-1}\huaB(\alpha(x),\alpha(y))&=&\huaB(x,y),\\
  \label{hom-hessian-2}\huaB(x\cdot y,\alpha(z))-\huaB(\alpha(x),y\cdot z)&=&\huaB(y\cdot x,\alpha(z))-\huaB(\alpha(y),x\cdot z),\quad \forall x,y,z\in A.
\end{eqnarray}
\end{defi}
Let $A$ be a vector space, for all $\huaB \in Sym^2(A^\ast)$, the linear map $\huaB^\sharp:A \longrightarrow A^\ast$ is given by
\begin{equation}\label{hessian}
\langle \huaB^\sharp(x),y\rangle=\huaB(x,y),\quad \forall x,y\in A.
\end{equation}
\begin{thm}
Let $(A,\cdot,\alpha)$ be a Hom-pre-Lie algebra with a Hessian structure $\huaB$. Then there exists a compatible Hom-L-dendriform algebra structure on $(A,\cdot,\alpha)$ given by
\begin{equation}
\huaB(x\triangleright y,z)=-\huaB(y,[\alpha^{-1}(x),\alpha^{-2}(z)]), \quad \huaB(x\triangleleft y,z)=-\huaB(y,\alpha^{-2}(z)\cdot\alpha^{-1}(x)),  \quad \forall x,y,z\in A.
\end{equation}
\end{thm}
\begin{proof}
By \eqref{hom-hessian-1} and \eqref{hessian}, we obtain that $(\huaB^\sharp)^{-1}\circ (\alpha^{-1})^\ast=\alpha\circ (\huaB^\sharp)^{-1}.$ Thus, we have
\begin{equation}\label{hessian-operator1}
(\huaB^\sharp)^{-1}\circ (\alpha^{-1})^\ast\circ (\alpha^{-1})^\ast=\alpha\circ (\huaB^\sharp)^{-1}\circ (\alpha^{-1})^\ast.
\end{equation}
For all $x,y,z\in A, \xi,\eta,\gamma\in A^*$, set $x=\alpha((\huaB^\sharp)^{-1}(\xi)), y=\alpha((\huaB^\sharp)^{-1}(\eta)), z=\alpha((\huaB^\sharp)^{-1}(\gamma))$, we have
\begin{eqnarray*}&&\langle (\huaB^\sharp)^{-1}((\alpha^{-1})^*(\xi))\cdot (\huaB^\sharp)^{-1}((\alpha^{-1})^*(\eta))-(\huaB^\sharp)^{-1}(\alpha^{-1})^*(\ad^\star((\huaB^\sharp)^{-1}(\alpha^{-1})^*(\alpha^*(\xi)))\eta\\
&&-R^\star((\huaB^\sharp)^{-1}(\alpha^{-1})^*(\alpha^*(\eta)))\xi),(\alpha^{-2})^*(\gamma)\rangle\\
&=&\langle x\cdot y-(\huaB^\sharp)^{-1}(\alpha^{-1})^*(\ad^\star_{(\huaB^\sharp)^{-1}(\xi)}\eta-R^\star_{(\huaB^\sharp)^{-1}(\eta)}\xi),(\alpha^{-2})^*(\gamma)\rangle \\
&=&\langle x \cdot y,\huaB^\sharp(\alpha(z))\rangle-\langle\ad^\star_{\alpha^{-1}(x)}\huaB^\sharp(\alpha^{-1}(y))-R^\star_{\alpha^{-1}(y)}\huaB^\sharp(\alpha^{-1}(x)),z \rangle \\
&=&\huaB(\alpha(z),x\cdot y)+\langle \huaB^\sharp(\alpha^{-1}(y)),[\alpha^{-2}(x),\alpha^{-2}(z)]\rangle-\langle \huaB^\sharp(\alpha^{-1}(x)),\alpha^{-2}(z)\cdot \alpha^{-2}(y)\rangle\\
&=&\huaB(\alpha(z),x\cdot y)+\huaB([\alpha^{-2}(x),\alpha^{-2}(z)],\alpha^{-1}(y))-\huaB(\alpha^{-1}(x),\alpha^{-2}(z)\cdot \alpha^{-2}(y))\\
&=&\huaB(\alpha(z),x\cdot y)+\huaB([x,z],\alpha(y))-\huaB(\alpha(x),z\cdot y)\\
&=&0,
\end{eqnarray*}
which implies that
\begin{eqnarray}\label{Hessian-o-operator-2}
\nonumber &&(\huaB^\sharp)^{-1}((\alpha^{-1})^*(\xi))\cdot (\huaB^\sharp)^{-1}((\alpha^{-1})^*(\eta))\\
 &=&(\huaB^\sharp)^{-1}(\alpha^{-1})^*(\ad^\star((\huaB^\sharp)^{-1}(\alpha^{-1})^*(\alpha^*(\xi)))\eta-R^\star((\huaB^\sharp)^{-1}(\alpha^{-1})^*(\alpha^*(\eta)))\xi).
\end{eqnarray}
 By \eqref{hessian-operator1} and \eqref{Hessian-o-operator-2}, we deduce that $(\huaB^\sharp)^{-1}\circ (\alpha^{-1})^*$ is a Hom-$\huaO$-operator associated to the representation $(A^*,(\alpha^{-1})^*,\ad^\star,-R^\star)$. By Theorem \ref{Operator-L}, there is a compatible Hom-L-dendriform algebra structure on $A$ defined by
\begin{eqnarray*}\huaB(x\triangleright y,z)
&=&\huaB((\huaB^\sharp)^{-1}(\alpha^{-1})^*(\ad^\star_{\alpha^{-1}(x)}\alpha^*(\huaB^\sharp(y))),z)\\
&=&\langle \ad^\star_{\alpha^{-1}(x)}\alpha^*(\huaB^\sharp(y)),\alpha^{-1}(z)\rangle\\
&=&-\langle \huaB^\sharp(y),[\alpha^{-1}(x),\alpha^{-2}(z)]\rangle\\
&=&-\huaB(y,[\alpha^{-1}(x),\alpha^{-2}(z)]).
\end{eqnarray*}
Similarly, we have $\huaB(x\triangleleft y,z)=-\huaB(y,\alpha^{-2}(z)\cdot\alpha^{-1}(x))$. The proof is finished.
\end{proof}
Next we consider the semi-direct product Hom-pre-Lie algebra $A\ltimes_{(\rho^{\star}-\mu^{\star},-\rho^{\star})}V^*$. Any linear map $T:V\longrightarrow A$ can be view as an element $\bar{T}\in \otimes^2(A\oplus V^*)$ via
\begin{equation}
\bar{T}(\xi+u,\eta+v)=\langle T(u),\eta\rangle,  \quad \forall  \xi+u,\eta+v\in A^*\oplus V,
\end{equation}
then $r=\bar{T}+\sigma(\bar{T})$ is symmetric.
\begin{thm}\label{semi-product}
Let $(A,\cdot,\alpha)$ be a Hom-pre-Lie algebra, $(V,\beta,\rho,\mu)$ be a representation of $(A,\cdot,\alpha)$ and $T:V\longrightarrow A$ be a linear map satisfying $T\circ \beta=\alpha\circ T$. Then $r=\bar{T}+\sigma(\bar{T})$ is a Hom-$\frks$-matrix in the Hom-pre-Lie algebra $A\ltimes_{(\rho^{\star}-\mu^{\star},-\rho^{\star})}V^*$ if and only if $T\circ \beta$ is a Hom-$\huaO$-operator.
\end{thm}
\begin{proof}
Let $\{v_1,\dots,v_n\}$ be a basis of $V$ and $\{v_1^*,\dots,v_n^*\}$ be its dual basis. It is obvious that $\bar{T}$ can be expressed by $\bar{T}=T(v_i)\otimes v_i^*$. Here the Einstein summation convention is used. Therefore, we can write $r=T(v_i)\otimes v_i^*+v_i^*\otimes T(v_i)$. Then we have
\begin{eqnarray*}r_{12} \cdot r_{23}
&=&-\alpha(T(v_i))\otimes \mu^\star(T(v_j))v_i^*\otimes(\beta^{-1})^*(v_j^*)+(\beta^{-1})^*(v_i^*)\otimes T(v_i)\cdot T(v_j)\otimes(\beta^{-1})^*(v_j^*)\\
&&+(\beta^{-1})^*(v_i^*)\otimes \rho^\star(T(v_i))v_j^*\otimes \alpha(T(v_j))-(\beta^{-1})^*(v_i^*)\otimes \mu^\star(T(v_i))v_j^*\otimes \alpha(T(v_j)),\\
-r_{23} \cdot r_{12}
&=&-\alpha(T(v_j))\otimes \rho^\star(T(v_i))v_j^*\otimes(\beta^{-1})^*(v_i^*)+\alpha(T(v_j))\otimes \mu^\star(T(v_i))v_j^*\otimes(\beta^{-1})^*(v_i^*)\\
&&-(\beta^{-1})^*(v_j^*)\otimes T(v_i)\cdot T(v_j) \otimes(\beta^{-1})^*(v_i^*)+(\beta^{-1})^*(v_j^*)\otimes \mu^\star(T(v_j))v_i^*\otimes \alpha(T(v_i)),\\
-r_{13} \cdot r_{12}
&=&-T(v_i)\cdot T(v_j)\otimes (\beta^{-1})^*(v_j^*)\otimes(\beta^{-1})^*(v_i^*)-\rho^\star(T(v_i))v_j^* \otimes\alpha(T(v_j))\otimes (\beta^{-1})^*(v_i^*)\\
&&+\mu^\star(T(v_i))v_j^* \otimes\alpha(T(v_j))\otimes (\beta^{-1})^*(v_i^*)+ \mu^\star(T(v_j))v_i^*\otimes(\beta^{-1})^*(v_j^*)\otimes \alpha(T(v_i)),\\r_{13} \cdot r_{23}
&=&-\alpha(T(v_i))\otimes (\beta^{-1})^*(v_j^*)\otimes\mu^\star(T(v_j))v_i^*+(\beta^{-1})^*(v_i^*)\otimes\alpha(T(v_j))\otimes \rho^\star(T(v_i))v_j^* \\
&&-(\beta^{-1})^*(v_i^*)\otimes\alpha(T(v_j))\otimes \mu^\star(T(v_i))v_j^*+(\beta^{-1})^*(v_i^*)\otimes(\beta^{-1})^*(v_j^*)\otimes-T(v_i)\cdot T(v_j).
\end{eqnarray*}
By direct computations, we have
\begin{eqnarray*}[[r,r]]
&=&\langle(\beta^{-1})^*(v_i^*),v_m\rangle v_m^*\otimes T(v_i)\cdot T(v_j)\otimes \langle(\beta^{-1})^*(v_j^*),v_n\rangle v_n^*\\
&&+\langle(\beta^{-1})^*(v_i^*),v_m\rangle v_m^*\otimes \langle \rho^\star(T(v_i))v_j^*,v_n\rangle v_n^*\otimes \alpha(T(v_j))\\
&&-\alpha(T(v_j))\otimes \langle \rho^\star(T(v_i))v_j^*,v_m\rangle v_m^*\otimes \langle(\beta^{-1})^*(v_i^*),v_n\rangle v_n^*\\
&&-\langle(\beta^{-1})^*(v_j^*),v_m\rangle v_m^*\otimes T(v_i)\cdot T(v_j)\otimes \langle(\beta^{-1})^*(v_i^*),v_n\rangle v_n^*\\
&&-T(v_i)\cdot T(v_j)\otimes \langle(\beta^{-1})^*(v_j^*),v_m\rangle v_m^*\otimes\langle(\beta^{-1})^*(v_i^*),v_n\rangle v_n^*\\
&&-\langle \rho^\star(T(v_i))v_j^*,v_m\rangle v_m^*\otimes \alpha(T(v_j))\otimes\langle(\beta^{-1})^*(v_i^*),v_n\rangle v_n^*\\
&&+\langle \mu^\star(T(v_i))v_j^*,v_m\rangle v_m^*\otimes \alpha(T(v_j))\otimes\langle(\beta^{-1})^*(v_i^*),v_n\rangle v_n^*\\
&&+\langle \mu^\star(T(v_j))v_i^*,v_m\rangle v_m^*\otimes\langle(\beta^{-1})^*(v_j^*),v_n\rangle v_n^*\otimes\alpha(T(v_i))\\
&&-\alpha(T(v_i))\otimes\langle(\beta^{-1})^*(v_j^*),v_m\rangle v_m^*\otimes \langle \mu^\star(T(v_j))v_i^*,v_n\rangle v_n^*\\
&&+\langle(\beta^{-1})^*(v_i^*),v_m\rangle v_m^*\otimes \alpha(T(v_j))\otimes\langle \rho^\star(T(v_i))v_j^*,v_n\rangle v_n^*\\
&&-\langle(\beta^{-1})^*(v_i^*),v_m\rangle v_m^*\otimes \alpha(T(v_j))\otimes\langle \mu^\star(T(v_i))v_j^*,v_n\rangle v_n^*\\
&&+\langle(\beta^{-1})^*(v_i^*),v_m\rangle v_m^*\otimes \langle(\beta^{-1})^*(v_j^*),v_n\rangle v_n^*\otimes T(v_i)\cdot T(v_j)\\
&=&v_m^*\otimes\langle v_i^*,\beta^{-1}(v_m)\rangle\langle v_j^*,\beta^{-1}(v_n)\rangle T(v_i)\cdot T(v_j)\otimes v_n^*\\
&&-v_m^*\otimes v_n^*\otimes \langle v_i^*,\beta^{-1}(v_m)\rangle\langle v_j^*,\rho(T(\beta^{-1}(v_i)))\beta^{-2}(v_n)\rangle\alpha(T(v_j))\\
&&+\langle v_j^*,\rho(T(\beta^{-1}(v_i)))\beta^{-2}(v_m)\rangle\langle v_i^*,\beta^{-1}(v_n)\rangle \alpha(T(v_j))\otimes v_m^*\otimes v_n^*\\
&&-v_m^*\otimes\langle v_j^*,\beta^{-1}(v_m)\rangle\langle v_i^*,\beta^{-1}(v_n)\rangle T(v_i)\cdot T(v_j)\otimes  v_n^*\\
&&-\langle v_j^*,\beta^{-1}(v_m)\rangle\langle v_i^*,\beta^{-1}(v_n)\rangle T(v_i)\cdot T(v_j)\otimes v_m^*\otimes v_n^*\\
&&+v_m^*\otimes\langle v_j^*,\rho(T(\beta^{-1}(v_i)))\beta^{-2}(v_m)\rangle\langle v_i^*,\beta^{-1}(v_n)\rangle\alpha(T(v_j))\otimes v_n^*\\
&&-v_m^*\otimes\langle v_j^*,\mu(T(\beta^{-1}(v_i)))\beta^{-2}(v_m)\rangle\langle v_i^*,\beta^{-1}(v_n)\rangle\alpha(T(v_j))\otimes v_n^*\\
&&-v_m^*\otimes v_n^*\otimes\langle v_i^*,\mu(T(\beta^{-1}(v_j)))\beta^{-2}(v_m)\rangle\langle v_j^*,\beta^{-1}(v_n)\rangle\alpha(T(v_i))\\
&&+\langle v_j^*,\beta^{-1}(v_m)\rangle\langle v_i^*,\mu(T(\beta^{-1}(v_j)))\beta^{-2}(v_n)\rangle\alpha(T(v_i))\otimes v_m^*\otimes v_n^*\\
&&-v_m^*\otimes\langle v_i^*,\beta^{-1}(v_m)\rangle\langle v_j^*,\rho(T(\beta^{-1}(v_i)))\beta^{-2}(v_n)\rangle\alpha(T(v_j))\otimes v_n^*\\
&&+v_m^*\otimes\langle v_i^*,\beta^{-1}(v_m)\rangle\langle v_j^*,\mu(T(\beta^{-1}(v_i)))\beta^{-2}(v_n)\rangle\alpha(T(v_j))\otimes v_n^*\\
&&+v_m^*\otimes v_n^*\otimes\langle v_i^*,\beta^{-1}(v_m)\rangle\langle v_j^*,\beta^{-1}(v_n)\rangle T(v_i)\cdot T(v_j)\\
&=&v_m^*\otimes T(\beta^{-1}(v_m))\cdot T(\beta^{-1}(v_n))\otimes v_n^*-v_m^*\otimes v_n^*\otimes T\beta(\rho(T(\beta^{-2}(v_m)))\beta^{-2}(v_n))\\
&&+T\beta(\rho(T(\beta^{-2}(v_n)))\beta^{-2}(v_m))\otimes v_m^*\otimes v_n^*- v_m^*\otimes T(\beta^{-1}(v_n))\cdot T(\beta^{-1}(v_m))\otimes v_n^*\\
&&-T(\beta^{-1}(v_n))\cdot T(\beta^{-1}(v_m))\otimes v_m^*\otimes v_n^*+v_m^*\otimes  T\beta(\rho(T(\beta^{-2}(v_n)))\beta^{-2}(v_m))\otimes v_n^*\\
&&-v_m^*\otimes T\beta(\mu(T(\beta^{-2}(v_n)))\beta^{-2}(v_m))\otimes v_n^*-v_m^*\otimes v_n^*\otimes T\beta(\mu(T(\beta^{-2}(v_n)))\beta^{-2}(v_m))\\
&&+T\beta(\mu(T(\beta^{-2}(v_m)))\beta^{-2}(v_n))\otimes v_m^*\otimes v_n^*-v_m^*\otimes T\beta(\rho(T(\beta^{-2}(v_m)))\beta^{-2}(v_n))\otimes v_n^*\\
&&+v_m^*\otimes T\beta(\mu(T(\beta^{-2}(v_m)))\beta^{-2}(v_n))\otimes v_n^*+v_m^*\otimes v_n^*\otimes T(\beta^{-1}(v_m))\cdot T(\beta^{-1}(v_n))\\
&=&v_m^*\otimes (T(\beta^{-1}(v_m))\cdot T(\beta^{-1}(v_n))-T\beta(\rho(T(\beta^{-2}(v_m)))\beta^{-2}(v_n)
-\mu(T(\beta^{-2}(v_n)))\beta^{-2}(v_m))\otimes v_n^*\\
&&+v_m^*\otimes v_n^*\otimes (T(\beta^{-1}(v_m))\cdot T(\beta^{-1}(v_n))-T\beta(\rho(T(\beta^{-2}(v_m)))\beta^{-2}(v_n)
-\mu(T(\beta^{-2}(v_n)))\beta^{-2}(v_m))\\
&&-(T(\beta^{-1}(v_n))\cdot T(\beta^{-1}(v_m))-T\beta(\rho(T(\beta^{-2}(v_n)))\beta^{-2}(v_m)
-\mu(T(\beta^{-2}(v_m)))\beta^{-2}(v_n))\otimes v_m^*\otimes v_n^*\\
&&-v_m^*\otimes (T(\beta^{-1}(v_n))\cdot T(\beta^{-1}(v_m))-T\beta(\rho(T(\beta^{-2}(v_n)))\beta^{-2}(v_m)
-\mu(T(\beta^{-2}(v_m)))\beta^{-2}(v_n))\otimes v_n^*,
\end{eqnarray*}
which implies that $[[r,r]]=0$ if and only if for all $u,v\in V$,
\begin{equation}
T\beta(\beta^{-2}(u))\cdot T\beta(\beta^{-2}(v))=T\beta(\rho(T\beta(\beta^{-1}(\beta^{-2}(u))))\beta^{-2}(v)+\mu(T\beta(\beta^{-1}(\beta^{-2}(v))))\beta^{-2}(u)).
\end{equation}
Furthermore, since $T\circ \beta=\alpha\circ T$, it is obvious that $T\circ \beta$ satisfies $(T\circ \beta)\circ \beta=\alpha \circ(T\circ \beta)$. Thus, $r$ is a  Hom-$\frks$-matrix in the Hom-pre-Lie algebra $A\ltimes_{(\rho^{\star}-\mu^{\star},-\rho^{\star})}V^*$ if and only if $T\circ \beta$ is a Hom-$\huaO$-operator.
\end{proof}
\begin{cor}
Let $(A,\triangleright,\triangleleft,\alpha)$ be a Hom-L-dendriform algebra. Then $r=v_i\otimes v_i^*+v_i^*\otimes v_i$ is a Hom-$\frks$-matrix in the associated vertical Hom-pre-Lie algebra $A\ltimes_{(L_\triangleright^*+L_\triangleleft^*,L_\triangleleft^*)}A^*.$
\end{cor}\begin{proof}
By Proposition \ref{rep} and Proposition \ref{dual-rep}, $(A^*,(\alpha^{-1})^*,L_\triangleright^*+L_\triangleleft^*,L_\triangleleft^*)$ is a dual representation of the  associated vertical Hom-pre-Lie algebra $(A,\cdot,\alpha)$. Moreover, $\alpha=\Id \circ \alpha:A\longrightarrow A$ is a Hom-$\huaO$-operator associated to the representation $(A,\alpha,L_\triangleright,-L_\triangleleft)$. By Theorem \ref{semi-product}, $r=v_i\otimes v_i^*+v_i^*\otimes v_i$ is a Hom-$\frks$-matrix.
\end{proof}

 \end{document}